\documentclass{article}
\usepackage[utf8]{inputenc}
\usepackage[english]{babel}

\usepackage{caption}
\usepackage{float}
\restylefloat{table}

\usepackage[colorinlistoftodos]{todonotes}
\usepackage{xpatch}
\usepackage[explicit]{titlesec}

\def\thissectiontitle{}
\def\thissectionnumber{}
\def\thissubsectiontitle{}
\def\thissubsectionnumber{}

\newtoggle{todoSection}
\newtoggle{todoSubsection}

\titleformat{\section}
  {\normalfont\Large\bfseries\sffamily}
  {\thesection}
  {0.5em}
  {\gdef\thissectiontitle{#1}\gdef\thissectionnumber{\thesection}#1}

\titleformat{\subsection}
  {\normalfont\large\bfseries\sffamily}
  {\thesubsection}
  {0.5em}
  {\gdef\thissubsectiontitle{#1}\gdef\thissubsectionnumber{\thesubsection}#1}

  \pretocmd{\section}{\global\toggletrue{todoSection}}{}{}
  \pretocmd{\subsection}{\global\toggletrue{todoSubsection}}{}{}

\AtBeginDocument{%
  \xpretocmd{\todo}{%
    \iftoggle{todoSubsection}{
     \addtocontents{tdo}{\protect\contentsline{subsection}%
        {\protect\numberline{\thissubsectionnumber}{\thissubsectiontitle}}{}{} }
      \global\togglefalse{todoSubsection}
        }{}
    }{}{}%
  }
\AtBeginDocument{%
  \xpretocmd{\todo}{%
    \iftoggle{todoSection}{
     \addtocontents{tdo}{\protect\contentsline{section}%
        {\protect\numberline{\thissectionnumber}{\thissectiontitle}}{}{} }
      \global\togglefalse{todoSection}
        }{}
    }{}{}%
  }

\usepackage{amsmath, amsthm, amssymb, amsfonts}
\usepackage{thmtools}
\usepackage{mathtools}

\usepackage{tikz-cd}
\usetikzlibrary{nfold}
\usepackage{quiver}

\usepackage{verbatim}

\usepackage[hypertexnames=false]{hyperref}
\usepackage[capitalize,noabbrev]{cleveref}

\usepackage{enumitem}

\theoremstyle{definition}
\newtheorem{defn}{Definition}[section]

\theoremstyle{definition}
\newtheorem*{defn*}{Definition}

\theoremstyle{definition}

\theoremstyle{plain}
\newtheorem{lem}[defn]{Lemma}

\theoremstyle{plain}

\theoremstyle{plain}
\newtheorem{cor}[defn]{Corollary}

\theoremstyle{plain}
\newtheorem{thm}[defn]{Theorem}

\theoremstyle{plain}
\newtheorem*{thm*}{Theorem}

\theoremstyle{plain}
\newtheorem{prop}[defn]{Proposition}

\theoremstyle{plain}
\newtheorem{exmpl}[defn]{Example}

\theoremstyle{plain}
\newtheorem*{exmpl*}{Example}

\theoremstyle{remark}
\newtheorem{rmk}[defn]{Remark}

\theoremstyle{remark}
\newtheorem*{rmk*}{Remark}

\theoremstyle{plain}
\newtheorem{introthm}{Theorem}

\theoremstyle{plain}

\theoremstyle{plain}
\newtheorem{introcor}[introthm]{Corollary}

\theoremstyle{plain}

\theoremstyle{plain}

\crefname{lem}{Lemma}{Lemmas}
\crefname{prop}{Proposition}{Propositions}
\crefname{defn}{Definition}{Definitions}
\crefname{introconj}{Conjecture}{Conjectures}
\crefname{cor}{Corollary}{Corollaries}

\newcommand{\Z}{\mathbb{Z}}

\newcommand{\Fp}{\mathbb{F}_p}

\newcommand{\op}[1]{#1^{\operatorname{op}}}
\newcommand{\id}[1]{\operatorname{id}_{#1}}

\newcommand{\cartsymb}{\arrow[dr, phantom,"\scalebox{1.5}{\color{black}$\lrcorner$}", near start, color=black]}

\renewcommand{\lim}{\operatorname{lim}}

\DeclareFontFamily{U}{min}{}
\DeclareFontShape{U}{min}{m}{n}{<-> udmj30}{}

\newcommand{\Alg}[1]{\operatorname{Alg}_{#1}}
\newcommand{\Mod}[1]{\operatorname{Mod}_{#1}}
\newcommand{\ModDisc}[1]{\operatorname{Mod}_{#1}^{\heartsuit}}
\newcommand{\QCoh}[1]{\operatorname{QCoh}(#1)}
\newcommand{\Frob}[2]{\operatorname{Frob}(#1, #2)}
\newcommand{\FrobSingle}[1]{\operatorname{Frob}(#1)}
\newcommand{\Leq}[2]{\operatorname{LEq}(#1,#2)}
\newcommand{\End}[1]{\operatorname{End}(#1)}

\newcommand{\Sp}{\operatorname{Sp}}
\renewcommand{\S}{\mathbb{S}}
\newcommand{\fgt}[1]{\operatorname{fgt}_{#1}}
\newcommand{\free}[1]{\operatorname{free}_{#1}}

\newcommand{\Cat}[1]{\mathcal{#1}}

\newcommand{\Map}[3]{\operatorname{Map}_{#1}({#2}, {#3})} 
\newcommand{\map}[3]{\operatorname{map}_{#1}({#2}, {#3})} 
\newcommand{\Fun}[2]{\operatorname{Fun}({#1},{#2})}

\newcommand{\CatInf}{\operatorname{Cat}_{\infty}}
\newcommand{\CatInfHuge}{\widehat{\operatorname{Cat}}_{\infty}}
\newcommand{\Groth}{\operatorname{Groth}}

\newcommand{\PrL}{\mathcal{P}r^{L}}
\newcommand{\PrSt}{\mathcal{P}r^{L}_{st}}

\newcommand{\PreSt}{\mathrm{PreSt}^L}

\newcommand{\zar}{\operatorname{zar}}
\newcommand{\frob}{\operatorname{frob}}

\newcommand{\loccit}{\textit{loc.~cit.~}}

\newcommand{\OX}{\mathcal{O}_X}

\newcommand\noloc{%
  \nobreak
  \mspace{6mu plus 1mu}
  {:}
  \nonscript\mkern-\thinmuskip
  \mathpunct{}
  \mspace{2mu}
}

\date{\today}
\author{Klaus Mattis\footnote{\href{http://www.klaus-mattis.com}{www.klaus-mattis.com}}\, and Timo Weiß\footnote{\href{mailto:timo.weiss@uni-mainz.de}{timo.weiss@uni-mainz.de}}}
\usepackage[margin=1in]{geometry}

\usepackage{tocloft}

\cftsetindents{section}{0em}{2em}
\title{The derived \texorpdfstring{$\infty$}{infinity}-category of Frobenius modules}

\begin{document}

\maketitle 

\begin{abstract}
    We prove that for $X$ a quasi-compact $\Fp$-scheme with affine diagonal (e.g.\ $X$ quasi-compact 
    and separated) there is a t-exact equivalence 
    $\Cat D(\Frob{\QCoh{X}}{F_*}) \to \Frob{\Cat D(\QCoh{X})}{\Cat D(F_*)}$
    of stable $\infty$-categories.
    Here, $\Frob{-}{-}$ denotes the $\infty$-category of generalized Frobenius modules 
    as introduced in \cite{cartmod}.
    This generalizes our result from \cite{cartmod},
    where we proved the above for regular Noetherian $\Fp$-schemes.
    As a byproduct we prove that the derived $\infty$-category of Frobenius (and Cartier) modules
    satisfies Zariski descent.
\end{abstract}

\hypersetup{pdfborder=0 0 0}
\tableofcontents
\hypersetup{pdfborder=1 1 1}
\newpage


\section{Introduction}

A crucial aspect of algebraic geometry over a field of positive characteristic $p > 0$ is the presence of the Frobenius endomorphism. 
In particular, the study of modules with an action of the Frobenius has lead to deep structural results.

If $X$ is an $\Fp$-scheme, then a (quasi-coherent) $\OX$-module with a left action of the absolute Frobenius $F\colon X \to X$ is called a Frobenius module. 
Equivalently, a Frobenius module is given by a pair $(M, \kappa_M)$ of an $\OX$-module $M$
and an $\OX$-linear morphism $\kappa_M \colon M \to F_*M$.
Examples of Frobenius modules include the structure sheaf $\OX$ and local cohomology modules,
both equipped with their natural left action by the Frobenius.
They have been used to prove important finiteness results, cf.\ \cite{hartshornespeiser, lyubeznik1997f}.
Moreover, Frobenius modules are related via a Riemann--Hilbert-type correspondence to $p$-torsion étale sheaves, 
cf.\ \cite{emerton2004riemann, bhatt2017riemannhilbertcorrespondencepositivecharacteristic,bockle2014cohomological}. 

There is also a dual notion of Cartier modules, which are related to Frobenius modules via Grothendieck--Serre duality by a result of Baudin \cite[Theorem 4.2.7]{baudin2025dualitycartiercrystalsperverse}. 
They are (quasi-coherent) $\OX$-modules with a right action of the Frobenius. Their category was first considered by Anderson in \cite{Anderson} and more thoroughly studied by Blickle and Böckle in \cite{blickle2011cartier,blickle2013cartier}. 
Both Frobenius modules and Cartier modules are important tools in positive characteristic commutative algebra and algebraic geometry,
and have been studied by many authors \cite{lyubeznik2001commutation,schwede2009f,patakfalvi2016frobeniustechniquesbirationalgeometry,99bc2f14-cf95-3c44-a77d-e80ab3dcf24f}.

In our previous article \cite{cartmod}, we introduced an $\infty$-categorical framework
for Cartier and Frobenius modules.
Recall that for any {($\infty$-)category} $\Cat C$ with an endofunctor $G \colon \Cat C \to \Cat C$,
we defined $\Frob{\Cat C}{G} \coloneqq \mathrm{LEq}(\id{\Cat C}, G)$ as the lax equalizer 
of the identity with $G$, cf.\ \cite[Definition 2.4]{cartmod}, and similarly 
$\mathrm{Cart}(\Cat C,G) \coloneqq \mathrm{LEq}(G, \id{\Cat C})$.
If $X$ is an $\Fp$-scheme for some prime $p$, 
$\Cat C = \QCoh{X}$ and $G = F_*$ is the pushforward along the absolute Frobenius,
one recovers the classical categories of Frobenius and Cartier modules on $X$, see the explanation above or e.g.\ \cite[Remark 1.3.2]{bhatt2017riemannhilbertcorrespondencepositivecharacteristic} 
and \cite[Definition 2.1]{blickle2011cartier} for definitions of these categories. 

In the following, we write $\Cat D(\Cat A) \in \PrSt$ 
for the (presentable stable) derived $\infty$-category of a Grothendieck abelian category $\Cat A$.
In \cite{cartmod} we proved the following:
\begin{thm*}[{\cite[Corollaries 6.2 and 6.3]{cartmod}}]
    Let $X$ be an $\Fp$-scheme. Then there is a canonical t-exact equivalence of presentable stable $\infty$-categories 
    \begin{equation*}
        \Cat D(\mathrm{Cart}(\QCoh{X}, F_*)) \xrightarrow{\simeq} \mathrm{Cart}(\Cat D(\QCoh{X}), \Cat D(F_*)).
    \end{equation*}
    If $X$ is moreover regular Noetherian, then there is a canonical t-exact equivalence of presentable stable $\infty$-categories 
    \begin{equation*}
        \Cat D(\mathrm{Frob}(\QCoh{X}, F_*)) \xrightarrow{\simeq} \mathrm{Frob}(\Cat D(\QCoh{X}), \Cat D(F_*)).
    \end{equation*}
\end{thm*}

Note that for this theorem to hold true, it is crucial to work with the derived $\infty$-category.
In fact, the proclaimed equivalences in the theorem do not hold if we replace the derived $\infty$-category by the ordinary derived category $h\Cat D(-)$:
In this case, the target categories would not even admit a triangulated structure, as the category $\Fun{\Delta^1}{h\Cat D(\QCoh{X})}$ does not do so. 

We needed the assumption that $X$ is regular Noetherian so that 
the Frobenius $F$ is flat by Kunz' theorem \cite[Theorem 2.1]{kunz1969characterizations},
and hence $F^*$ is an exact functor. If $X$ is arbitrary,
then this is no longer the case, and the proof strategy of \cite{cartmod}
does no longer work.
The goal of this paper is a version of the above result,
where we essentially remove the regularity and Noetherian hypothesis.
For this, we need the following definition:
\begin{defn*}
    Let $X$ be a scheme. We say that $X$ is \emph{geometric} 
    if $X$ is quasi-compact and has affine diagonal (the latter is also called 
    \emph{semi-separated} in the literature).
\end{defn*}
\begin{exmpl*}
    Since closed immersions are affine, and affine morphisms are quasi-compact,
    we see that quasi-compact separated schemes are geometric,
    and geometric schemes are qcqs.
    In particular, any affine scheme is geometric.
\end{exmpl*}

\begin{introthm}[{\cref{lem:ab4:main-thm}}] \label{lem:intro:main-thm}
    Let $X$ be a geometric $\Fp$-scheme.
    Then the canonical map 
    \begin{equation*}
        \Cat D(\Frob{\QCoh{X}}{F_*}) \to \Frob{\Cat D(\QCoh{X})}{\Cat D(F_*)}
    \end{equation*}
    is a t-exact equivalence of presentable stable $\infty$-categories.
\end{introthm}

\subsection*{Outline of the proof}
We now explain how we prove \cref{lem:intro:main-thm}.
Both sides of the proclaimed equivalence are left-complete (this follows from \cref{lem:ab4:frob-qcoh-is-ab4n,lem:ab4:qcoh-is-ab4n,lem:ab4:conservative-limit-preserving-reflects-left-complete}),
hence (using \cref{lem:ab4:plus-equivalence}) it suffices to show the following theorem:
\begin{introthm}[{\cref{lem:zariski:main-thm}}] \label{lem:intro:main-thm-plus}
    Let $X$ be a geometric $\Fp$-scheme.
    Then the canonical map 
    \begin{equation*}
        \Cat D^+(\Frob{\QCoh{X}}{F_*}) \to \Frob{\Cat D^+(\QCoh{X})}{\Cat D^+(F_*)}
    \end{equation*}
    is a t-exact equivalence of stable $\infty$-categories.
    Here, $\Cat D^+$ denotes the full subcategory of bounded above objects.
\end{introthm}
In \cref{lem:zariski:left-sep-derived-category-sheaf} and (the proof of) \cref{lem:zariski:main-thm} we show that both sides of the equivalence are Zariski sheaves in $X$,
and that there is a canonical morphism between them, cf.\ \cref{lem:zariski:canonical-map}.
At this point, we heavily use that we are working with the derived $\infty$-category instead of the ordinary derived category.
Hence, it suffices to prove the statement for $X = \mathrm{Spec}(R)$ an affine scheme.
On affines, it boils down to the Schwede--Shipley recognition theorem for stable module categories \cite[Theorem 7.1.2.1]{higheralgebra},
which is an $\infty$-categorical version of Gabriel's theorem \cite[Corollaire V.1.1]{gabriel}:
If one has a stable $\infty$-category $\Cat C$ with a compact projective generator $M$,
then $\Cat C$ is equivalent to the stable $\infty$-category of $A$-module spectra,
where $A = \End{M}$ is the endomorphism ring spectrum of $M$.
Hence, in order to show \cref{lem:intro:main-thm-plus} on affines,
it therefore suffices to show that both categories admit a compact projective generator
(for this see \cref{lem:affine:DFrob-cp-gen,lem:affine:FrobD-cp-gen}),
and that the endomorphism ring spectra of the generators are equivalent (\cref{lem:affine:eta-equivalence}).
In fact, as expected, the endomorphism ring spectra on both sides 
are equivalent to the discrete (non-commutative) ring $\op{R[F]}$ of e.g.\ \cite[Definition 3.2.1]{bockle2014cohomological}.

\begin{rmk*}
    Note that we cannot directly show that both sides of \cref{lem:intro:main-thm} 
    are Zariski sheaves, as we a priori do not know that they are left-complete.
    Indeed, for the proof of the left-completeness of $\Cat D(\Frob{\QCoh{X}}{F_*})$,
    we need to know that $\Frob{\QCoh{X}}{F_*}$ satisfies a weaker version of axiom
    $\mathrm{AB}4^*$. In order to show this, we first establish 
    \cref{lem:intro:main-thm-plus}, as this axiom only depends on the bounded above objects 
    of the derived $\infty$-category, and hence can then be deduced from the fact that already 
    $\QCoh{X}$ satisfies it.
\end{rmk*}

\subsection*{Zariski descent}
As already mentioned above, one step in the proof of the main theorem is a reduction to affine schemes
via Zariski descent. In particular, we get the following result:
\begin{introthm} [{\cref{lem:zariski:left-sep-derived-category-sheaf}}]
    Let $X$ be a geometric $\Fp$-scheme and write $X_{\zar}$ for the small (quasi-compact) Zariski site of $X$, cf.\ \cref{defn:zariski:site}.
    Then $\Cat D^+(\Frob{\QCoh{-}}{F_*})$ defines a Zariski sheaf on $X_{\zar}$ with values in $\CatInf$.
\end{introthm}
As explained above, all involved $\infty$-categories are already left-complete.
Hence, one also gets via an analogous proof (using that then 
already $\Cat D(-)$ preserves certain limits as 
in the proof of \cref{lem:zariski:descent-on-derived},
see \cite[Proposition A.4.23]{heyer20246}) the following:
\begin{introthm}
    Let $X$ be a geometric $\Fp$-scheme.
    Then $\Cat D(\Frob{\QCoh{-}}{F_*})$ defines a Zariski sheaf on $X_{\zar}$ 
    with values in $\PrSt$.
\end{introthm}
Essentially dualizing the proof, we also get 
\begin{introcor}
    Let $X$ be a geometric $\Fp$-scheme. 
    Then $\Cat D(\mathrm{Cart}(\QCoh{-},F_*))$ defines a Zariski sheaf on $X_{\zar}$ with values in $\PrSt$.
\end{introcor}
\begin{rmk*}
    One can use similar techniques to show stronger descent statements for the derived 
    $\infty$-categories of Frobenius and Cartier modules.
\end{rmk*}

\subsection*{Use of the \texorpdfstring{$\infty$}{infinity}-category of spectra}
As explained above, our proof reduces to the affine case $X = \mathrm{Spec}(R)$,
and then identifies both sides with the category of $R[F]$-module spectra,
i.e.\ working over the absolute base $\S \in \Sp$.
This is mostly for convenience of reference, as the $\infty$-categorical Schwede--Shipley 
theorem is formulated in this setting.
Since all our categories are (limits of) derived categories of rings and schemes,
they are all $\Z$-linear. Hence, one could avoid the $\infty$-category of spectra 
by proving a $\Z$-linear version of Schwede--Shipley,
which then would identify both sides of the equivalence with the category 
$\Mod{R[F]}(\Cat D(\Z))$ 
(which in turn is of course equivalent to $\Mod{R[F]}(\Sp)$,
since the $\mathbb{E}_1$-morphism $\S \to R[F]$ canonically factors through $H\Z$,
and $\Cat D(\Z) = \Mod{H\Z}(\Sp)$).

\subsection*{Notations and Conventions}
This article is written in the language of $\infty$-categories, as developed by Lurie in \cite{highertopoi,higheralgebra,sag}.
We fix four universes, small, large, huge and very huge.
By default, i.e.\ if not specified otherwise, any $\infty$-category will be large,
and any scheme will be small. We will employ the following notations:
\begin{table}[H] 
\centering
\begin{tabular}{l|l}
    $p$ &A prime number \\
    $\CatInf$ &The huge $\infty$-category of large $\infty$-categories \\
    $\CatInfHuge$ &The very huge $\infty$-category of huge $\infty$-categories \\
    $\PrL$ &The huge $\infty$-category of large presentable $\infty$-categories, \\& with small colimit-preserving functors \\
    $\PrSt$&The huge $\infty$-category of large presentable stable $\infty$-categories, \\& with small colimit-preserving functors \\
    $\Groth$ & The huge $\infty$-category of large Grothendieck abelian categories, \\& with small colimit-preserving exact functors \\
    $\Sp$ &The stable $\infty$-category of small spectra \\
    $\S \in \Sp$& The sphere spectrum 
\end{tabular}
\end{table}
In particular, note that $\Sp \in \PrSt$,
and that $\CatInf, \PrL, \PrSt, \Groth \in \CatInfHuge$.

We always use homological notation for a t-structure on a stable $\infty$-category.
If $\Cat D$ is a stable $\infty$-category with a t-structure $(\Cat D_{\ge 0}, \Cat D_{\le 0})$,
we write $H \colon \Cat D^\heartsuit \hookrightarrow \Cat D$ 
for the inclusion of the heart.

If $\Cat D$ is a stable $\infty$-category, it is in particular 
enriched in the $\infty$-category of spectra $\Sp$, cf.\ \cite[Example 7.4.14]{gepnerhaugseng}.
We write $\map{\Cat D}{-}{-} \colon \op{\Cat D} \times \Cat D \to \Sp$
for the mapping spectrum. If $L \colon \Cat D \rightleftarrows \Cat E \noloc R$
is an adjunction between stable $\infty$-categories,
we will use without mention that it upgrades to a spectrally enriched adjunction.
In particular, there is a natural equivalence $\map{\Cat E}{L-}{-} \cong \map{\Cat D}{-}{R-}$.

\subsection*{Acknowledgments}
We thank our respective advisors Tom Bachmann and Manuel Blickle for many helpful discussions. We 
thank Anton Engelmann for reading a draft of this article. We thank the anonymous referee for pointing out a mistake 
in an earlier version of this article.

The authors acknowledge support by the Deutsche Forschungsgemeinschaft 
(DFG, German Research Foundation) through the Collaborative Research 
Centre TRR 326 \textit{Geometry and Arithmetic of Uniformized 
Structures}, project number 444845124.

\section{Functoriality of Frobenius modules}

In this section, we show that the construction of (generalized) Frobenius modules from \cite[Definition 2.4]{cartmod} is functorial
in the $\infty$-category and its endofunctor. 
Moreover, we prove that the resulting functor is limit-preserving, cf.\ \cref{lem:functoriality:frob-functor-limits,lem:functoriality:frob-functor-on-groth-prst-limits}.

\begin{defn}
    We define a functor $\End{-} \colon \CatInfHuge \to \CatInfHuge$
    as the pullback in the following diagram in $\Fun{\CatInfHuge}{\CatInfHuge}$:
    \begin{center}
        \begin{tikzcd}
            \End{-} \cartsymb \ar[d, "\kappa"'] \ar[r, "U"] &\id{\CatInfHuge} \ar[d, "{(\id{}, \id{})}"] \\
            \Fun{\Delta^1}{-} \ar[r, "{(s,t)}"'] &\id{\CatInfHuge} \times \id{\CatInfHuge} \rlap{.}
        \end{tikzcd}
    \end{center}
    If $\Cat C$ is a (huge) $\infty$-category,
    we call $\End{\Cat C}$ the \emph{$\infty$-category of endomorphisms in $\Cat C$}.
\end{defn}

\begin{rmk}
    Let $\Cat C$ be a huge $\infty$-category.
    Then $\End{\Cat C} \cong \Leq{\id{\Cat C}}{\id{\Cat C}}$ is
    the lax equalizer of the identity with the identity on $\Cat C$, see e.g.\ \cite[Definition 2.1]{cartmod} for a definition.
    In other words, it fits in a pullback square in $\CatInfHuge$:
    \begin{center}
        \begin{tikzcd}
            \End{\Cat C} \cartsymb \ar[r, "{U_{\Cat C}}"] \ar[d, "{\kappa_{\Cat C}}"'] &\Cat C \ar[d, "{(\id{\Cat C}, \id{\Cat C})}"] \\
            \Fun{\Delta^1}{\Cat C} \ar[r, "{(s, t)}"'] &\Cat C \times \Cat C\rlap{.}
        \end{tikzcd}
    \end{center}
    This holds since limits of functors are computed pointwise.
    A similar statement is true for morphisms, see also the diagram in the next lemma.
\end{rmk}

\begin{lem} \label{lem:functoriality:end-limit-preserving}
    Let $F \colon \Cat C \to \Cat D$ be a functor of huge $\infty$-categories (i.e.\ a morphism in $\CatInfHuge$).
    Moreover, let $p\colon K \to \End{\Cat C}$ be a diagram in $\End{\Cat C}$ admitting a limit and assume that $F$ preserves the limit of the diagram $U_{\Cat C} \circ p$.
    Then $\End{F} \colon \End{\Cat C} \to \End{\Cat D}$ preserves the limit of the diagram $p$. 
    In particular, if $F$ preserves limits, then so does $\End{F}$.
\end{lem}
\begin{proof}
    By construction (as limits of functors are computed pointwise),
    we know that $\End{F}$ is the dashed diagonal arrow in the following diagram:
    \begin{center}
        \begin{tikzcd}
            \End{\Cat C} \ar[dr, dashed, "\End{F}"]\ar[rrr, "U_{\Cat C}"] \ar[ddd, "\kappa_{\Cat C}"'] &&& \Cat C \ar[ddd, "{(\id{}, \id{})}"] \ar[dl, "F"] \\
            &\End{\Cat D} \cartsymb \ar[d, "\kappa_{\Cat D}"'] \ar[r, "U_{\Cat D}"] &\Cat D \ar[d, "{(\id{}, \id{})}"] & \\
            &\Fun{\Delta^1}{\Cat D} \ar[r, "{(s,t)}"'] &\Cat D \times \Cat D & \\
            \Fun{\Delta^1}{\Cat C} \ar[ur, "F_*"] \ar[rrr, "{(s,t)}"'] &&& \Cat C \times \Cat C \ar[ul, "F \times F"] \rlap{.}
        \end{tikzcd}
    \end{center}
    Thus, since $U_{\Cat C}$ and $U_{\Cat D}$ are limit-preserving and conservative functors by \cite[Proposition 2.6 (b) and (c)]{cartmod},
    and since $F$ preserves the limit of the diagram $U_{\Cat C} \circ p$, the lemma follows from the commutativity of the top quadrilateral.
\end{proof}

We now construct the $\infty$-category of (generalized) Frobenius modules 
as a functor $\End{\CatInf} \to \CatInf$.
\begin{prop} \label{lem:functoriality:frob-functor}
    There is a functor $\Frob{-}{-} \colon \End{\CatInf} \to \CatInf$
    such that for every endofunctor $F \colon \Cat C \to \Cat C$,
    we have that $\Frob{\Cat C}{F}$ is the category of generalized Frobenius modules 
    from \cite[Definition 2.4]{cartmod}.

    Moreover, on morphisms $f \colon (\Cat C, F) \to (\Cat D, G)$ the functor $\FrobSingle{f}$ is given by the dashed morphism induced by the following diagram:
    \begin{center}
        \begin{tikzcd}
            \Frob{\Cat C}{F} \ar[dr, dashed] \ar[rr, "{U_{\Cat C}}"] \ar[dd, "{\kappa_{\Cat C}}"'] &&\Cat C \ar[d, "f"] \\
            & \Frob{\Cat D}{G} \cartsymb \ar[r, "{U_{\Cat D}}"]\ar[d, "{\kappa_{\Cat D}}"'] &\Cat D \ar[d, "{(\id{\Cat D},G)}"] \\
            \Fun{\Delta^1}{\Cat C} \ar[r, "f_*"']& \Fun{\Delta^1}{\Cat D} \ar[r, "{(s,t)}"'] &\Cat D \times \Cat D \rlap{.}
        \end{tikzcd}
    \end{center}
    Here, the commutativity of the outer diagram is witnessed by
    \begin{equation*}
        \id{\Cat D} fU_{\Cat C} \cong fs\kappa_{\Cat C} \cong s f_* \kappa_{\Cat C}
    \end{equation*}
    and
    \begin{equation*}
        GfU_{\Cat C} \cong fFU_{\Cat C} \cong ft\kappa_{\Cat C} \cong t f_* \kappa_{\Cat C}.
    \end{equation*}
\end{prop}
\begin{proof}
    Write $U_{\CatInf} \colon \End{\CatInf} \to \CatInf$ and $\kappa_{\CatInf} \colon \End{\CatInf} \to \Fun{\Delta^1}{\CatInf}$ for the canonical forgetful functors
    (i.e.\ the projections out of the limit).
    By adjunction (and the identification of $s \circ \kappa_{\CatInf} \cong U_{\CatInf} \cong t \circ \kappa_{\CatInf}$ via the 
    defining diagram of $\End{\CatInf}$), this can be seen as a natural transformation $\kappa_{\CatInf} \colon U_{\CatInf} \to U_{\CatInf}$.
    We let $\Frob{-}{-}$ be the pullback in the diagram of functors in $\Fun{\End{\CatInf}}{\CatInf}$:
    \begin{center}
        \begin{tikzcd}
            \Frob{-}{-} \cartsymb \ar[r] \ar[d] &U_{\CatInf} \ar[d, "{(\id{}, \kappa_{\CatInf})}"] \\
            \Fun{\Delta^1}{U_{\CatInf}} \ar[r, "{(s,t)}"']& U_{\CatInf} \times U_{\CatInf} \rlap{.}
        \end{tikzcd}
    \end{center}
    If $(\Cat C, F) \in \End{\CatInf}$, then we see that (as limits of functors are computed pointwise) there is a pullback square
    \begin{center}
        \begin{tikzcd}
            \Frob{\Cat C}{F} \cartsymb \ar[r, "U_{\Cat C}"] \ar[d, "\kappa_{\Cat C}"'] &\Cat C \ar[d, "{(\id{\Cat C}, F)}"] \\
            \Fun{\Delta^1}{\Cat C} \ar[r, "{(s,t)}"'] &\Cat C \times \Cat C \rlap{,}
        \end{tikzcd}
    \end{center}
    i.e.\ we see that $\Frob{\Cat C}{F}$ agrees with the $\infty$-category 
    from \cite[Definition 2.4]{cartmod}.
    A similar argument gives us the description on morphisms.
\end{proof}

\begin{prop} \label{lem:functoriality:frob-functor-limits}
    The functor $\Frob{-}{-} \colon \End{\CatInf} \to \CatInf$ from \cref{lem:functoriality:frob-functor}
    preserves limits.
\end{prop}
\begin{proof}
    Since $\Frob{-}{-}$ is itself a limit of functors,
    it suffices to show that $U_{\CatInf}$ and $\Fun{\Delta^1}{-}$ preserve limits.
    The first statement is \cite[Proposition 2.6 (c)]{cartmod},
    and the second holds e.g.\ because $\Fun{\Delta^1}{-}$ is right adjoint to the functor $- \times \Delta^1 \colon \CatInf \to \CatInf$.
\end{proof}

\begin{lem} \label{lem:functoriality:frob-functor-fully-faithful}
    Let $f \colon (\Cat C, F) \to (\Cat D, G)$ be a morphism in $\End{\CatInf}$.
    If the underlying functor $f \colon \Cat C \to \Cat D$ is fully faithful,
    then also the induced functor $\Frob{\Cat C}{F} \to \Frob{\Cat D}{G}$ is fully faithful.
\end{lem}
\begin{proof}
    The induced functor is the limit of the fully faithful functors 
    ${f \colon \Cat C \to \Cat D}$, $f \times f \colon \Cat C^2 \to \Cat D^2$
    and $f_* \colon \Fun{\Delta^1}{\Cat C} \to \Fun{\Delta^1}{\Cat D}$,
    and thus itself fully faithful.
\end{proof}

\begin{prop} \label{lem:functoriality:frob-functor-on-groth-prst}
    Let $\Cat C$ be either $\Groth \in \CatInfHuge$ or $\PrSt \in \CatInfHuge$.
    Write $\iota \colon \Cat C \to \CatInf$ for the obvious inclusion.
    Then there exists the dashed functor making the following diagram commute:
    \begin{center}
        \begin{tikzcd}
            \End{\CatInf} \ar[rr, "{\Frob{-}{-}}"] &&\CatInf \\
            \End{\Cat C} \ar[u, "\End{\iota}"] \ar[rr, dashed] &&\Cat C \ar[u, "\iota"'] \rlap{.}
        \end{tikzcd}
    \end{center}
    We write again $\Frob{-}{-}$ for the dashed functor.
    In other words, the functor $\Frob{-}{-} \circ \End{\iota}$ factors over $\Cat C$. 
\end{prop}
\begin{proof}
    We first prove the result if $\Cat C = \Groth$.
    It suffices to show that for every $(\Cat A, F) \in \End{\Groth}$ 
    we have that $\Frob{\Cat A}{F}$ is again Grothendieck abelian,
    and that if $\phi \colon (\Cat A, F) \to (\Cat B, G) \in \End{\Groth}$ is a colimit-preserving exact functor
    commuting with $F$ and $G$, then also $\FrobSingle{\phi}$ is colimit-preserving and exact.

    The first statement is just \cite[Corollary 2.8 (i)]{cartmod}.
    For the second statement, note that by definition, there is a commutative diagram 
    \begin{center}
        \begin{tikzcd}
            \Frob{\Cat A}{F} \ar[r, "U_{\Cat A}"] \ar[d, "\FrobSingle{\phi}"'] &\Cat A \ar[d, "\phi"] \\
            \Frob{\Cat B}{G} \ar[r, "U_{\Cat B}"'] &\Cat B \rlap{.}
        \end{tikzcd}
    \end{center}
    Now note that $U_{\Cat A}$ and $U_{\Cat B}$ are conservative, colimit-preserving and exact functors by \cite[Corollary 2.8 (b) and (i)]{cartmod},
    and that $\phi$ is colimit-preserving and exact by assumption.
    Thus, the commutativity of the above square immediately implies the result.

    If $\Cat C = \PrSt$, one now uses \cite[Corollary 2.8 (e) and (f)]{cartmod}
    to see that on objects $\Frob{-}{-}$
    lands again in $\PrSt$. The argument for morphisms is analogous to the case $\Cat C = \Groth$,
    again using \cite[Corollary 2.8 (b) and (f)]{cartmod}.
\end{proof}

\begin{lem} \label{lem:functoriality:frob-functor-on-groth-prst-limits}
    Let $\Cat C$ be either $\Groth \in \CatInfHuge$ or $\PrSt \in \CatInfHuge$.
    Then the functor $\Frob{-}{-} \colon \End{\Cat C} \to \Cat C$ 
    from \cref{lem:functoriality:frob-functor-on-groth-prst}
    preserves limits.
\end{lem}
\begin{proof}
    By definition, there is a commutative diagram 
    \begin{center}
        \begin{tikzcd}
            \End{\CatInf} \ar[rr, "{\Frob{-}{-}}"] &&\CatInf \\
            \End{\Cat C} \ar[u, "\End{\iota}"] \ar[rr, "\Frob{-}{-}"'] &&\Cat C \ar[u, "\iota"'] \rlap{.}
        \end{tikzcd}
    \end{center}
    The top functor preserves limits by \cref{lem:functoriality:frob-functor-limits}.
    If $\iota$ preserves limits, then so does $\End{\iota}$ by \cref{lem:functoriality:end-limit-preserving}.
    Thus, using the commutativity of the diagram, it is enough to show that $\iota$ is conservative and preserves limits.
    Conservativity is clear (as equivalences are always colimit-preserving and exact).
    For limits, in the case of $\Cat C = \Groth$ this is \cite[Proposition C.5.4.21]{sag},
    and if $\Cat C = \PrSt$ this follows from \cite[Proposition 4.8.2.18]{higheralgebra} and \cite[Proposition 5.5.3.13]{highertopoi}.
\end{proof}
 
\section{Frobenius modules as modules over a ring}
In this section, we show that \cref{lem:intro:main-thm} holds for affine schemes, cf.\ \cref{lem:affine:main-thm}.
To do so, we identify both relevant $\infty$-categories as $\infty$-categories of module spectra over equivalent $\mathbb{E}_1$-rings.
For this, we use the Schwede--Shipley theorem, which we recall in \cref{lem:affine:schwede-shipley}.

\begin{defn}
    Let $S$ be an associative discrete ring.
    We write $\ModDisc{S}$ for the abelian (ordinary) category of (right)
    $S$-modules.
\end{defn}

\begin{defn}
    Let $S$ be an $\mathbb{E}_1$-ring spectrum (e.g.\ a discrete ring).
    Then we write $\Mod{S}$ for the presentable stable $\infty$-category of (right) $S$-module spectra,
    cf.\ \cite[Definition 7.1.1.2]{higheralgebra}.
    See \cite[Corollary 7.1.1.5]{higheralgebra} and \cite[Corollary 4.2.3.7]{higheralgebra} for a proof that $\Mod{S}$ is stable and presentable.
\end{defn}

\begin{lem} \label{lem:affine:free-forget}
    Let $S$ be an $\mathbb{E}_1$-ring spectrum.
    There is an adjunction
    \begin{equation*}
        \free{S} \colon \Sp \rightleftarrows \Mod{S} \noloc \fgt{S}
    \end{equation*}
    such that $\fgt{S}\free{S} \cong S \otimes_{\S} -$,
    where $\S$ denotes the sphere spectrum,
    and $\free{S} \S \cong S$.
    Moreover, $\fgt{S}$ is conservative and preserves colimits.

    If $S$ is an $\mathbb{E}_\infty$-ring spectrum,
    then the right-hand side can be equipped with a symmetric monoidal 
    structure such that $\free{S}$ is symmetric monoidal.
\end{lem}
\begin{proof}
    That there is an adjunction with the required description 
    of the composition was shown in \cite[Corollary 4.2.4.8]{higheralgebra}.
    We will defer the proof that $\free{S} \S \cong S$ to the end.
    That $\fgt{S}$ is conservative can be seen as follows: By stability,
    we only have to see that if $\fgt{S}(M) \cong 0$, then $M \cong 0$.
    This is true as there is only the trivial $S$-module structure on $0$.
    In order to see that $\fgt{S}$ also preserves colimits,
    it suffices to show that it has a right adjoint,
    which follows from \cite[Remark 4.2.3.8]{higheralgebra},
    since the tensor product on $\Sp$ commutes with colimits in both variables.

    If $S$ is an $\mathbb{E}_\infty$-ring spectrum,
    then we equip $\Mod{S}$ with the symmetric monoidal structure
    from \cite[Theorem 4.5.2.1]{higheralgebra} (here we use that 
    the tensor product in $\Sp$ commutes with colimits in each variable).
    It was shown essentially in \cite[Corollary 5.1.2.6]{higheralgebra},
    that the unit map $\S \to S$ induces a symmetric monoidal functor 
    $\Sp \to \Mod{S}$. As this functor is given by base change, it is clear that this functor 
    is equivalent to $\free{S}$, which gives us a symmetric monoidal structure on the latter.

    Similarly, if $S$ was only assumed to be $\mathbb{E}_1$,
    then \cite[Corollary 5.1.2.6]{higheralgebra} shows that $\Sp \to \Mod{S}$
    is $\mathbb{E}_0$-monoidal, i.e.\ preserves the unit object.
    This is exactly the claim that $\free{S} \S \cong S$.
\end{proof}

If $S$ is a connective $\mathbb{E}_1$-ring spectrum,
then we equip $\Mod{S}$ with the accessible t-structure from \cite[Proposition 7.1.1.13]{higheralgebra}.
The following is immediate from the definition:
\begin{lem} \label{lem:affine:forget-t-exact}
    Let $S$ be a connective $\mathbb{E}_1$-ring spectrum.
    Then the forgetful functor $\fgt{S} \colon \Mod{S} \to \Sp$
    is t-exact.
\end{lem}
 
\begin{lem} \label{lem:affine:maps-from-unit}
    Let $S$ be an $\mathbb{E}_1$-ring spectrum.
    Then there is an equivalence of functors
    \begin{equation*}
        \map{\Mod{S}}{S}{-} \cong \fgt{S}.
    \end{equation*}
\end{lem}
\begin{proof}
    By \cref{lem:affine:free-forget} we have $\free{S} \S \cong S$.
    Thus, by adjunction we get
    \begin{equation*}
        \map{\Mod{S}}{S}{-} \cong \map{\Mod{S}}{\free{S} \S}{-} \cong \map{\Sp}{\S}{\fgt{S} -} \cong \fgt{S},
    \end{equation*}
    where we used in the last equivalence that $\S$ is the unit and $\map{\Sp}{-}{-}$ the internal mapping object
    of the category of spectra.
\end{proof}

Recall the following lemma:
\begin{lem} [{\cite[Theorem 7.1.2.13]{higheralgebra}}] \label{lem:affine:derived-cat-of-modules}
    Let $R$ be a commutative discrete ring.
    Then there is a symmetric monoidal equivalence
    \begin{equation*}
        \Cat{D}(\ModDisc{R}) \cong \Mod{R}.
    \end{equation*}
\end{lem}

\begin{lem}
    There is a functor
    \begin{equation*}
        \Theta \colon \Alg{\mathbb{E}_1}(\Sp) \to {(\PrSt)}_{\Mod{\S}/}
    \end{equation*}
    that sends an $\mathbb{E}_1$-algebra $R$ to the functor $\Mod{\S} \to \Mod{R}$ given by base change.
\end{lem}
\begin{proof}
    See the discussion right before \cite[Corollary 4.8.5.13]{higheralgebra} with $\Cat C = \Sp$,
    together with the canonical equivalence $\Sp \cong \Mod{\S}$.
\end{proof}

\begin{rmk}
    An object $F \colon \Mod{\S} \to \Cat C$ in ${(\PrSt)}_{\Mod{\S}/}$
    is the same data as a presentable stable $\infty$-category $\Cat C$, together with an object $M = F(\S) \in \Cat C$.
    This is true because $\Mod{\S} \cong \Sp$, and a colimit-preserving exact functor $\Sp \to \Cat C$
    is just the data of an object (the image of $\S$).

    If $R$ is an $\mathbb{E}_1$-algebra, then $\Theta(R) \cong (\Mod{R}, R)$,
    as $\S \otimes_{\S} R \cong R$.
\end{rmk}

This functor has the following property:
\begin{prop}[Schwede--Shipley theorem] \label{lem:affine:schwede-shipley}
    The functor $\Theta$ is fully faithful and has a right adjoint.
    An object $(\Cat C, M) \in {(\PrSt)}_{\Mod{\S}/}$ is in the essential image 
    if and only if $M$ is a stable compact generator of $\Cat C$ in the sense of \cref{defn:cpgen:stable-generator}.
    Moreover, the right adjoint of $\Theta$ is given by the functor
    that sends $(\Cat C, M)$ to $\End{M}$, 
    where $\End{M}$ denotes the endomorphism spectrum of $M$ equipped with its $\mathbb{E}_1$-ring structure.
\end{prop}
\begin{proof}
    That $\Theta$ is fully faithful and admits a right adjoint follows from \cite[Theorem 4.8.5.11]{higheralgebra},
    as $\Theta$ is the restriction of $\widehat{\Theta}_*$,
    whereas the description of the essential image is (the proof of) \cite[Theorem 7.1.2.1]{higheralgebra}.
    The description of the right adjoint is \cite[Remark 4.8.5.12]{higheralgebra}.
\end{proof}

For the rest of this section, let $R$ be a (discrete) $\Fp$-algebra.
Write $F \colon R \to R$ for the Frobenius endomorphism,
i.e.\ the ring morphism given by $x \mapsto x^p$.

Note that the functor $F_* \colon \ModDisc{R} \to \ModDisc{R}$ is a colimit-preserving 
exact functor of Grothendieck abelian categories, as it has both a left adjoint \cite[\href{https://stacks.math.columbia.edu/tag/05DQ}{Tag  05DQ}]{stacks-project} and a right adjoint \cite[\href{https://stacks.math.columbia.edu/tag/08YP}{Tag 08YP}]{stacks-project}.
In particular, the functor $\Cat D(F_*) \colon \Cat{D}(\ModDisc{R}) \to \Cat{D}(\ModDisc{R})$ 
exists, cf.\ \cite[Notation A.3]{cartmod}.

\begin{prop} \label{lem:affine:adjunction-on-module-cats}
    The Frobenius pushforward $\Cat D(F_*) \colon \Cat{D}(\ModDisc{R}) \to \Cat{D}(\ModDisc{R})$
    has a symmetric monoidal left adjoint $\mathbf{L}F^* \colon \Cat{D}(\ModDisc{R}) \to \Cat{D}(\ModDisc{R})$.
\end{prop}
\begin{proof}
    We have a canonical symmetric monoidal equivalence $\Cat{D}(\ModDisc{R}) \cong \Mod{R}$ by \cref{lem:affine:derived-cat-of-modules}.
    Hence, the existence of the left adjoint is \cite[Proposition 4.6.2.17]{higheralgebra}.
    That the left adjoint is symmetric monoidal follows essentially from \cite[Remark 4.5.3.2]{higheralgebra}.
\end{proof}

In the following we consider the category $\Frob{\Cat D(\ModDisc{R})}{\Cat D(F_*)}$ of (generalized) Frobenius modules. Recall that it is defined by the pullback square
\begin{center}
    \begin{tikzcd}
        \Frob{\Cat D(\ModDisc{R})}{\Cat D(F_*)} \cartsymb \ar[d, "\kappa_{\Cat D}"'] \ar[r, "U_{\Cat D}"] & \Cat D(\ModDisc{R}) \ar[d, "{(\id{},\Cat D(F_*))}"] \\
        \Fun{\Delta^1}{\Cat D(\ModDisc{R})} \ar[r, "{(s, t)}"'] & \Cat D(\ModDisc{R})^2\rlap{.}
    \end{tikzcd}
\end{center}

\begin{prop} \label{lem:affine:frob-left-adjoint}
    The functor $U_{\Cat D} \colon \Frob{\Cat D(\ModDisc{R})}{\Cat D(F_*)} \to \Cat D(\ModDisc{R})$ has a left adjoint $L_{\Cat D}$.
    The underlying module of the left adjoint is given by $U_{\Cat D}L_{\Cat D} \cong \bigsqcup_{n \ge 0} (\mathbf{L}F^*)^n$.
\end{prop}
\begin{proof}
    Since by \cref{lem:affine:adjunction-on-module-cats} $\Cat D(F_*) \colon \Cat D(\ModDisc{R}) \to \Cat D(\ModDisc{R})$ admits a left adjoint 
    $\mathbf{L}F^*$,
    and since $\Cat D(\ModDisc{R})$ admits all countable coproducts (as it is presentable),
    it follows from \cite[Corollary 4.4]{cartmod} that $U_{\Cat D}$ has a left adjoint with the given description.
\end{proof}

Recall that the category $\Frob{\ModDisc{R}}{F_*}$ is defined by the pullback square
\begin{center}
    \begin{tikzcd}
        \Frob{\ModDisc{R}}{F_*} \cartsymb \ar[d, "\kappa_\heartsuit"'] \ar[r, "U_\heartsuit"] & \ModDisc{R} \ar[d, "{(\id{}, F_*)}"] \\
        \Fun{\Delta^1}{\ModDisc{R}} \ar[r, "{(s, t)}"'] & (\ModDisc{R})^2
    \end{tikzcd}
\end{center}
in $\CatInf$.
Since the functor $F_*$ admits a left adjoint (cf.\ \cite[\href{https://stacks.math.columbia.edu/tag/05DQ}{Tag 05DQ}]{stacks-project}) and the category $\ModDisc{R}$ admits all countable coproducts, the forgetful functor $U_\heartsuit$ also admits a left adjoint $L_\heartsuit\colon \ModDisc{R} \to \Frob{\ModDisc{R}}{F_*}$ by \cite[Corollary 4.4]{cartmod}.

By \Cref{lem:functoriality:frob-functor-on-groth-prst} the category $\Frob{\ModDisc{R}}{F_*}$ is Grothendieck abelian, so we can consider its derived $\infty$-category.
Note that the functors $U_\heartsuit$ and $\kappa_\heartsuit$ both are exact and colimit-preserving: For $U_\heartsuit$ this follows from \cite[Corollary 2.8 (i)]{cartmod} as $\ModDisc{R}$ is a Grothendieck abelian category and $F_*$ is exact and colimit-preserving, and for $\kappa_\heartsuit$ this can be checked after applying the source and target functors where it amounts to the fact that $U_\heartsuit$ and $F_*$ are exact and colimit-preserving.
Hence, they induce derived functors $\Cat D(U_\heartsuit)\colon \Cat D(\Frob{\ModDisc{R}}{F_*}) \to \Cat D(\ModDisc{R})$ and $\Cat D(\kappa_\heartsuit)\colon \Cat D(\Frob{\ModDisc{R}}{F_*}) \to \Cat D(\Fun{\Delta^1}{\ModDisc{R}})$ by \cite[Proposition A.2]{cartmod}.

\begin{defn}
    We define $\Phi \colon \Cat{D}(\Frob{\ModDisc{R}}{F_*}) \to \Frob{\Cat D(\ModDisc{R})}{\Cat D(F_*)}$
    to be the functor fitting into the following diagram as the dashed arrow, via the universal property of the pullback:
    \begin{center}
        \begin{tikzcd}
            \Cat{D}(\Frob{\ModDisc{R}}{F_*}) \ar[dd, "{\Cat D(\kappa_\heartsuit)}"'] \ar[dr, dashed, "\Phi"] \ar[drr, bend left, "{\Cat D(U_\heartsuit)}"] \\
            &\Frob{\Cat D(\ModDisc{R})}{\Cat D(F_*)} \cartsymb \ar[d, "\kappa_{\Cat D}"'] \ar[r, "U_{\Cat D}"] &\Cat D(\ModDisc{R}) \ar[d, "{(\id{}, \Cat D(F_*))}"] \\
            \Cat D(\Fun{\Delta^1}{\ModDisc{R}}) \ar[r, "\xi"']&\Fun{\Delta^1}{\Cat D(\ModDisc{R})} \ar[r, "{(s,t)}"'] &\Cat D(\ModDisc{R})^2 \rlap{.}
        \end{tikzcd}
    \end{center}
    Here, $\xi$ is the functor from the discussion before \cite[Theorem 5.1]{cartmod}.
    The commutativity of the outer solid diagram can be deduced in the exact same way 
    as the commutativity of the diagram in \cite[Theorem 5.1]{cartmod}, 
    see the paragraph directly behind the statement of the theorem.
\end{defn}
In the rest of this section, we show that $\Phi$ is an equivalence.

For this, we first introduce t-structures on the source and target of $\Phi$ such that $\Phi$ is t-exact with respect to these t-structures.

Recall that the derived $\infty$-category $\Cat D(\Cat A)$ of any Grothendieck abelian category $\Cat A$ carries a t-structure which is described in \cite[Definition 1.3.5.16]{higheralgebra}. 
In particular, the $\infty$-category $\Cat D(\Frob{\ModDisc{R}}{F_*})$ carries an induced t-structure as $\Frob{\ModDisc{R}}{F_*}$ is a Grothendieck abelian category by \Cref{lem:functoriality:frob-functor-on-groth-prst}.

Moreover, we have the following.

\begin{lem}
    Define the two full subcategories $\Frob{\Cat D(\ModDisc{R})}{\Cat D(F_*)}_{\geq 0}$ and 
    $\Frob{\Cat D(\ModDisc{R})}{\Cat D(F_*)}_{\leq 0}$ of 
    $\Frob{\Cat D(\ModDisc{R})}{\Cat D(F_*)}$ as follows: 
    For an object $M \in \Frob{\Cat D(\ModDisc{R})}{\Cat D(F_*)}$ 
    we have
    \begin{align*}
        &M \in \Frob{\Cat D(\ModDisc{R})}{\Cat D(F_*)}_{\geq 0} \text{ if and only if } U_{\Cat D}M \in \Cat D(\ModDisc{R})_{\geq 0}, \text{ and} \\
        &M \in \Frob{\Cat D(\ModDisc{R})}{\Cat D(F_*)}_{\leq 0} \text{ if and only if } U_{\Cat D}M \in \Cat D(\ModDisc{R})_{\leq 0}.
    \end{align*}
    These subcategories define a t-structure on $\Frob{\Cat D(\ModDisc{R})}{\Cat D(F_*)}$ and the forgetful functor $U_{\Cat D}$ is t-exact.
\end{lem}
\begin{proof}
    Since the functor $\Cat D(F_*)\colon \Cat D(\ModDisc{R}) \to \Cat D(\ModDisc{R})$ is t-exact by the definition of the derived functor, cf.\ \cite[Notation A.3]{cartmod},
    this follows directly from \cite[Proposition 3.3]{cartmod}.
\end{proof}

\begin{lem} \label{lem:affine:Phi-t-exact}
    The functor $\Phi$ is t-exact.
\end{lem}
\begin{proof}
    By the definition of the t-structure on $\Frob{\Cat D(\ModDisc{R})}{\Cat D(F_*)}$
    it suffices to show that $U_{\Cat D} \circ \Phi \cong \Cat D(U_{\heartsuit})$ is t-exact.
    This follows from the definition of the derived functor, cf.\ \cite[Notation A.3]{cartmod}.
\end{proof}

Therefore, $\Phi$ induces a functor $\Frob{\ModDisc{R}}{F_*} \to \Frob{\Cat D(\ModDisc{R})}{\Cat D(F_*)}^\heartsuit$ 
between the hearts. We proceed by showing that this induced functor is an equivalence.

\begin{lem} \label{lem:affine:heart-identification}
    Consider the functor 
    \begin{equation*}
        \Lambda \coloneqq \Frob{\pi_0}{h} \colon \Frob{\Cat D(\ModDisc{R})}{\Cat D(F_*)} \to \Frob{\ModDisc{R}}{F_*}
    \end{equation*}
    induced by the functoriality of $\Frob{-}{-}$ from \cref{lem:functoriality:frob-functor} and the commutative square $h$:
    \begin{center}
        \begin{tikzcd}
            \Cat D(\ModDisc{R}) \ar[d, "\pi_0"'] \ar[r, "\Cat D(F_*)"] &\Cat D(\ModDisc{R}) \ar[d, "\pi_0"] \\
            \ModDisc{R} \ar[r, "F_*"'] &\ModDisc{R} \rlap{,}
        \end{tikzcd}
    \end{center}
    which exists by \cite[Lemma A.4]{cartmod}.
    Then the composition of the inclusion of the heart 
    \begin{equation*}
        H \colon \Frob{\Cat D(\ModDisc{R})}{\Cat D(F_*)}^\heartsuit \hookrightarrow \Frob{\Cat D(\ModDisc{R})}{\Cat D(F_*)}
    \end{equation*}
    with $\Lambda$ 
    is an equivalence $\Frob{\Cat D(\ModDisc{R})}{\Cat D(F_*)}^\heartsuit \cong \Frob{\ModDisc{R}}{F_*}$.
\end{lem}
\begin{proof}
    This is done in the proof of \cite[Proposition 3.4]{cartmod}. Note that the functor which is called $\Phi$ in \loccit is exactly the composition
    \begin{equation*}
        \Frob{\Cat D(\ModDisc{R})}{\Cat D(F_*)}^\heartsuit \xhookrightarrow{H} \Frob{\Cat D(\ModDisc{R})}{\Cat D(F_*)} \xrightarrow{\Lambda} \Frob{\ModDisc{R}}{F_*}. \qedhere
    \end{equation*}
\end{proof}

\begin{lem} \label{lem:affine:kappa-heart}
    Under the identification of the heart of $\Frob{\Cat D(\ModDisc{R})}{\Cat D(F_*)}$ from \cref{lem:affine:heart-identification},
    the following diagram commutes:
    \begin{center}
        \begin{tikzcd}
            \Frob{\Cat D(\ModDisc{R})}{\Cat D(F_*)} \ar[d, "\pi_0"'] \ar[r, "\kappa_{\Cat D}"] & \Fun{\Delta^1}{\Cat D(\ModDisc{R})} \ar[d, "(\pi_0)_*"] \\
            \Frob{\ModDisc{R}}{F_*} \ar[r, "\kappa_{\heartsuit}"'] &\Fun{\Delta^1}{\ModDisc{R}} \rlap{.}
        \end{tikzcd}
    \end{center}
\end{lem}
\begin{proof}
    By the description of the equivalence of \cref{lem:affine:heart-identification} and of the t-structure on $\Fun{\Delta^1}{\Cat D(\ModDisc{R})}$, 
    we have to show that the following diagram commutes:
    \begin{center}
        \begin{tikzcd}
            \Frob{\Cat D(\ModDisc{R})}{\Cat D(F_*)} \ar[d, "\pi_0"'] \ar[r, "\kappa_{\Cat D}"] & \Fun{\Delta^1}{\Cat D(\ModDisc{R})} \ar[d, "\pi_0"] \\
            \Frob{\Cat D(\ModDisc{R})}{\Cat D(F_*)}^\heartsuit \ar[d, hook, "H"'] \ar[r, "\kappa_{\Cat D}^{\heartsuit}"] & \Fun{\Delta^1}{\Cat D(\ModDisc{R})}^\heartsuit \ar[d, hook, "H"] \\
            \Frob{\Cat D(\ModDisc{R})}{\Cat D(F_*)} \ar[d, "\Lambda"'] \ar[r, "\kappa_{\Cat D}"'] & \Fun{\Delta^1}{\Cat D(\ModDisc{R})} \ar[d, "(\pi_0)_*"] \\
            \Frob{\ModDisc{R}}{F_*} \ar[r, "\kappa_{\heartsuit}"'] &\Fun{\Delta^1}{\ModDisc{R}} \rlap{.}
        \end{tikzcd}
    \end{center}
    Note that the functor $\kappa_{\Cat D}$ is t-exact because both functors $s\kappa_{\Cat D} \cong U_{\Cat D}$ and $t\kappa_{\Cat D} \cong \Cat D(F_*)$ are.
    This implies that the restriction $\kappa_{\Cat D}^\heartsuit$ appearing in the diagram is well-defined and that the two upper squares commute.
    The lower square commutes by definition of $\Lambda$.
\end{proof}

\begin{lem} \label{lem:affine:Phi-heart-equivalence}
    Under the identification of the heart of $\Frob{\Cat D(\ModDisc{R})}{\Cat D(F_*)}$ from \cref{lem:affine:heart-identification},
    the functor 
    \begin{equation*}
        \Phi^\heartsuit \coloneqq \pi_0 \Phi H \colon \Frob{\ModDisc{R}}{F_*} \to \Frob{\ModDisc{R}}{F_*}
    \end{equation*}
    is equivalent to the identity.
\end{lem}
\begin{proof}
    We first show that there is an equivalence $\omega \colon \kappa_\heartsuit \Phi^\heartsuit \cong \kappa_\heartsuit$.
    It is given by the following chain of equivalences:
    \begin{align*}
        \kappa_\heartsuit \Phi^\heartsuit 
        &= \kappa_\heartsuit \pi_0 \Phi H 
        \cong (\pi_0)_* \kappa_{\Cat D} \Phi H 
        = (\pi_0)_* \xi \Cat D(\kappa_{\heartsuit}) H \\
        &= (\pi_0)_* \xi H \kappa_{\heartsuit} 
        = (\pi_0)_* H_* \kappa_{\heartsuit} 
        \cong \kappa_{\heartsuit}.
    \end{align*}
    Here, the equality signs are just definitions, whereas the equivalences 
    are given by \cref{lem:affine:kappa-heart} and since $(\pi_0)_* H_* \cong (\pi_0 H)_* \cong \id{}$.  

    Thus, we also get an equivalence $\eta \colon U_\heartsuit \Phi^\heartsuit \cong U_\heartsuit$ as follows:
    \begin{equation*}
        U_\heartsuit \Phi^\heartsuit \cong s \kappa_\heartsuit \Phi^\heartsuit \xrightarrow[\simeq]{\omega}  s \kappa_\heartsuit \cong U_\heartsuit.
    \end{equation*}
    In particular, by construction, $\omega$ and $\eta$ are compatible in the defining diagram of $\Frob{\ModDisc{R}}{F_*}$,
    and thus, by the universal property of the pullback, we get an equivalence $\Phi^\heartsuit \cong \id{\Frob{\ModDisc{R}}{F_*}}$.
\end{proof}

We now show that both relevant $\infty$-categories have a canonical stable compact generator (in the sense of \cref{defn:cpgen:stable-generator}).
Using the Schwede--Shipley theorem, this will also imply that we can identify the $\infty$-categories with certain module categories.
\begin{lem} \label{lem:affine:DFrob-cp-gen}
    The stable $\infty$-category $\Cat{D}(\Frob{\ModDisc{R}}{F_*})$
    has a stable compact generator $A \coloneqq HL_\heartsuit(R)$.
\end{lem}
\begin{proof}
    Using \cref{lem:cpgen:abelian-implies-stable}, it suffices to show that $L_\heartsuit(R)$ is a compact projective generator (in the sense of \cref{defn:cpgen:abelian-generator})
    of the Grothendieck abelian category $\Frob{\ModDisc{R}}{F_*}$.
    By \cite[Corollary 2.8]{cartmod}, the forgetful functor $U_{\heartsuit}$
    is a conservative and colimit-preserving functor between Grothendieck abelian categories.
    Then \cref{lem:cpgen:preservation} shows that $L_\heartsuit$ preserves compact projective generators.
    Since $R$ is a compact projective generator of $\ModDisc{R}$
    we see that $L_{\heartsuit}(R)$ is a compact projective generator of $\Frob{\ModDisc{R}}{F_*}$.
\end{proof}

\begin{cor} \label{lem:affine:DFrob-is-module-cat}
    There is a canonical equivalence of $\infty$-categories
    \begin{equation*}
        \Mod{\End{A}} \xrightarrow{\simeq} \Cat D(\Frob{\ModDisc{R}}{F_*}).
    \end{equation*}
\end{cor}
\begin{proof}
    There is a morphism
    \begin{equation*}
        (\Mod{\End{A}}, \End{A}) \to (\Cat D(\Frob{\ModDisc{R}}{F_*}), A)
    \end{equation*}
    in ${(\PrSt)}_{{\Mod{\S}}/}$ which is given by the counit of the adjunction from \cref{lem:affine:schwede-shipley}.
    The same result implies that the target of this morphism is in the essential image of $\Theta$,
    as $A$ is a stable compact generator of $\Cat D(\Frob{\ModDisc{R}}{F_*})$ by \cref{lem:affine:DFrob-cp-gen}.
    Using this, it follows from the fully faithfulness of $\Theta$ that the counit above
    is an equivalence in this case.
\end{proof}

\begin{lem} \label{lem:affine:frob-ring-is-discrete}
    The object $L_{\Cat D} HR \in \Frob{\Cat D(\ModDisc{R})}{\Cat D(F_*)}$ lies in the heart.
\end{lem}
\begin{proof}
    By definition of the induced t-structure it suffices to show that $U_{\Cat D}L_{\Cat D} HR$ lies in the heart.
    By \cref{lem:affine:frob-left-adjoint} we have $U_{\Cat D}L_{\Cat D}HR \cong \bigsqcup_{n \ge 0} (\mathbf{L}F^*)^n HR$.
    Since $\bigsqcup_{n \ge 0}$ preserves 
    connective objects (because connective objects are closed under colimits) and also coconnective objects (since the t-structure 
    is stable under filtered colimits, cf.\ \cite[Proposition 1.3.5.21]{higheralgebra}),
    it suffices to show that $(\mathbf{L}F^*)^n HR$ is discrete for all $n$.
    Since $\mathbf{L}F^*$ is symmetric monoidal, cf.\ \cref{lem:affine:adjunction-on-module-cats}, 
    and since $HR \in \Cat D(\ModDisc{R})$ is the tensor unit, we know that $\mathbf{L}F^* HR \cong HR$.
    Inductively, we thus see that $(\mathbf{L}F^*)^n HR \cong HR$, which is clearly in the heart.
\end{proof}

\begin{lem} \label{lem:affine:L-heart-id}
    Under the identification of the heart of $\Frob{\Cat D(\ModDisc{R})}{\Cat D(F_*)}$ from \cref{lem:affine:heart-identification},
    there is a canonical equivalence $L_\heartsuit \cong \pi_0 L_{\Cat D} H$ of functors $\ModDisc{R} \to \Frob{\ModDisc{R}}{F_*}$.
\end{lem}
\begin{proof}
    Note that the adjunction $L_{\Cat D} \dashv U_{\Cat D}$ 
    induces an adjunction $\pi_0 L_{\Cat D} H \dashv \pi_0 U_{\Cat D} H$ on the heart, cf.\ \cite[Proposition 1.3.17 (iii)]{bbd}.
    By uniqueness of left adjoints, it suffices to provide an equivalence $U_{\heartsuit} \cong \pi_0 U_{\Cat D} H$.
    Under the identification of the heart (via the functor $\Lambda H$),
    this is immediate from the definition of $\Lambda$.
\end{proof}

\begin{lem} \label{lem:affine:FrobD-cp-gen}
    The object $B \coloneqq \Phi(A)$ is a stable compact generator of $\Frob{\Cat D(\ModDisc{R})}{\Cat D(F_*)}$.
    Moreover, there is a canonical equivalence $B \cong L_{\Cat D}HR$.
\end{lem}
\begin{proof}
    We first establish the equivalence $B \cong L_{\Cat D}HR$.
    It is given by the following chain of equivalences:
    \begin{align*}
        \Phi(A) \cong \Phi H L_\heartsuit R \cong H \Phi^\heartsuit L_\heartsuit R \cong H L_\heartsuit R \cong H \pi_0 L_{\Cat D} H R \cong L_{\Cat D} H R.
    \end{align*}
    Here, the first equivalence is the definition of $A$,
    the second holds because $\Phi$ is t-exact by \cref{lem:affine:Phi-t-exact},
    and the third equivalence follows since $\Phi^\heartsuit \cong \id{}$ by \cref{lem:affine:Phi-heart-equivalence}.
    The fourth equivalence is just \cref{lem:affine:L-heart-id},
    whereas the last equivalence holds because $L_{\Cat D} HR$ is discrete by \cref{lem:affine:frob-ring-is-discrete}.

    Hence, it suffices to show that $L_{\Cat D}HR$ is a stable compact generator of $\Frob{\Cat D(\ModDisc{R})}{\Cat D(F_*)}$.
    Note that by \cite[Corollary 2.8]{cartmod}, the forgetful functor $U_{\Cat D}$
    is a conservative and colimit-preserving functor between presentable stable $\infty$-categories.
    Then \cref{lem:cpgen:preservation} shows that $L_{\Cat D}$ preserves stable compact generators.
    Since $HR$ is a stable compact generator of $\Cat D(\ModDisc{R})$ by \cref{lem:cpgen:abelian-implies-stable},
    we see that $L_{\Cat D}HR$ is a stable compact generator of $\Frob{\Cat D(\ModDisc{R})}{\Cat D(F_*)}$.
\end{proof}

\begin{cor} \label{lem:affine:FrobD-is-module-cat}
    There is a canonical equivalence of $\infty$-categories
    \begin{equation*}
        \Mod{\End{B}} \xrightarrow{\simeq} \Frob{\Cat D(\ModDisc{R})}{\Cat D(F_*)}.
    \end{equation*}
\end{cor}
\begin{proof}
    There is a morphism
    \begin{equation*}
        (\Mod{\End{B}}, \End{B}) \to (\Frob{\Cat D(\ModDisc{R})}{\Cat D(F_*)}, B)
    \end{equation*}
    in ${(\PrSt)}_{{\Mod{\S}}/}$ which is given by the counit of the adjunction from \cref{lem:affine:schwede-shipley}.
    The proof is now the same as the one of \cref{lem:affine:DFrob-is-module-cat},
    where we use that $B$ is a stable compact generator of $\Frob{\Cat D(\ModDisc{R})}{\Cat D(F_*)}$
    by \cref{lem:affine:FrobD-cp-gen}.
\end{proof}

Combining this corollary with \cref{lem:affine:DFrob-is-module-cat}, 
we see that the $\infty$-categories $\Cat D(\Frob{\ModDisc{R}}{F_*})$ and $\Frob{\Cat D(\ModDisc{R})}{\Cat D(F_*)}$ 
are equivalent if the $\mathbb{E}_1$-ring spectra $\End{A}$ and $\End{B}$ are equivalent.
Next, we show that the latter is indeed true.

\begin{lem} \label{lem:affine:A-B-heart}
    Both $A \in \Cat{D}(\Frob{\ModDisc{R}}{F_*})$
    and $B \in \Frob{\Cat D(\ModDisc{R})}{\Cat D(F_*)}$ 
    lie in the heart of the respective $\infty$-categories.
\end{lem}
\begin{proof}
    By definition, $A = HL_{\heartsuit} R$ is in the heart.
    Moreover, $B = \Phi(A)$ is in the heart since $\Phi$ is t-exact by \cref{lem:affine:Phi-t-exact}.
\end{proof}

\begin{lem} \label{lem:affine:mapping-spectrum-A-discrete}
    The mapping spectrum $\map{}{A}{A}$ is connective (in fact discrete, i.e.\ in $\Sp^\heartsuit$).
\end{lem}
\begin{proof}
    Note that $A$ lives in the heart by \Cref{lem:affine:A-B-heart}, and thus it is clear that $\map{}{A}{A}$ is coconnective.
    To see that it is also connective, note that we have 
    \begin{equation*}
        \pi_{-n}\map{}{A}{A} \cong \operatorname{Ext}^n(L_{\heartsuit} R, L_{\heartsuit} R),
    \end{equation*}
    cf.\ the discussion in \cite[Notation 1.1.2.17]{higheralgebra}.
    But now note that by (the proof of) \cref{lem:affine:DFrob-cp-gen} 
    we see that $L_{\heartsuit} R$ is a projective object.
    In particular, the $\operatorname{Ext}$-groups vanish for $n \neq 0$,
    which immediately implies connectivity of the mapping spectrum.
\end{proof}

\begin{lem} \label{lem:affine:mapping-spectrum-B-discrete}
    The mapping spectrum $\map{}{B}{B}$ is connective (in fact discrete, i.e.\ in $\Sp^\heartsuit$).
\end{lem}
\begin{proof}
    By \cref{lem:affine:FrobD-cp-gen} there is an equivalence $B \cong L_{\Cat D} HR$.
    Now we have equivalences 
    \begin{equation*}
        \map{}{B}{B} \cong \map{}{L_{\Cat D} HR}{L_{\Cat D} HR} \cong \map{}{HR}{U_{\Cat D} L_{\Cat D} HR} \cong \fgt{R} U_{\Cat D} L_{\Cat D} HR.
    \end{equation*}
    Here, the second equivalence holds by adjunction, whereas the last equivalence is (essentially) \cref{lem:affine:maps-from-unit}.
    Since $\fgt{R}$ is t-exact by \cref{lem:affine:forget-t-exact}, it suffices to show that $U_{\Cat D} L_{\Cat D} HR$ 
    is in the heart. This was shown in \cref{lem:affine:frob-ring-is-discrete}.
\end{proof}

\begin{prop} \label{lem:affine:eta-equivalence}
    The morphism of $\mathbb{E}_1$-rings $\End{\Phi} \colon \End{A} \to \End{B}$ (given by the action of $\Phi$ on the endomorphism spectra) is an equivalence.
\end{prop}
\begin{proof}
    It suffices to show that the underlying map is an equivalence, i.e.\ we have to see that the map 
    \begin{equation*}
        \map{\Cat{D}(\Frob{\ModDisc{R}}{F_*})}{A}{A} \xrightarrow{\Phi} \map{\Frob{\Cat D(\ModDisc{R})}{\Cat D(F_*)}}{B}{B}
    \end{equation*}
    is an equivalence.
    Note that by \cref{lem:affine:mapping-spectrum-A-discrete,lem:affine:mapping-spectrum-B-discrete} both mapping spectra are connective, hence it suffices to show that $\Phi$ 
    induces an equivalence of mapping spaces 
    \begin{equation*}
        \Map{\Cat{D}(\Frob{\ModDisc{R}}{F_*})}{A}{A} \xrightarrow{\Phi} \Map{\Frob{\Cat D(\ModDisc{R})}{\Cat D(F_*)}}{B}{B}.
    \end{equation*}
    Since both $A$ and $B$ are in the respective hearts, cf.\ \cref{lem:affine:A-B-heart},
    it follows that this map is actually induced by $\Phi^\heartsuit$.
    But $\Phi^\heartsuit$ is an equivalence by \cref{lem:affine:Phi-heart-equivalence}.
    This proves the proposition.
\end{proof}

\begin{thm} \label{lem:affine:main-thm}
    The functor $\Phi$ is an equivalence.
\end{thm}
\begin{proof}
    Consider the commutative diagram
    \begin{center}
        \begin{tikzcd}
            (\Cat{D}(\Frob{\ModDisc{R}}{F_*}), A) \ar[r, "\Phi"] &(\Frob{\Cat D(\ModDisc{R})}{\Cat D(F_*)}, B) \\
            (\Mod{\End{A}}, \End{A}) \ar[u] \ar[r, "\Theta(\End{\Phi})"'] & (\Mod{\End{B}}, \End{B}) \ar[u]
        \end{tikzcd}
    \end{center}
    in ${(\PrSt)}_{{\Mod{\S}}/}$,
    where the vertical maps are the counits of the adjunction from \cref{lem:affine:schwede-shipley}.
    These counits are equivalences by (the proofs of) \cref{lem:affine:DFrob-is-module-cat,lem:affine:FrobD-is-module-cat}.
    As $\End{\Phi}$ is an equivalence by \cref{lem:affine:eta-equivalence},
    the theorem follows.
\end{proof}

\begin{rmk}
    One can identify the $\mathbb{E}_1$-ring $\End{A}$.
    In fact, it is discrete by \cref{lem:affine:mapping-spectrum-A-discrete}.
    Hence (since $A = HL_{\heartsuit} R$), it is enough to compute the associative ring $\Map{\Frob{\ModDisc{R}}{F_*}}{L_\heartsuit R}{L_\heartsuit R}$.
    This ring can be identified (via Gabriels theorem \cite[Corollaire V.1.1]{gabriel}) with the ring $R[\tau]^{\mathrm{op}}$, see \cite[Definition 3.2.1]{bockle2014cohomological}.
    For this we use that $\Frob{\ModDisc{R}}{F_*}$ agrees with their category of $\tau$-sheaves, essentially by construction,
    and that they speak of left modules, whereas we work with right modules.
\end{rmk}

\section{Zariski descent of Frobenius modules}

In this section, we show that the functor $\Cat D^+(\Frob{\QCoh{-}}{F_*})$
defines a Zariski sheaf on the small (quasi-compact) Zariski site of a geometric scheme, cf.\ \cref{lem:zariski:left-sep-derived-category-sheaf}.
Combining this with our results from the previous section, we prove \cref{lem:intro:main-thm-plus}, cf.\ \cref{lem:zariski:main-thm}.

We start by recalling that the derived $\infty$-category is functorial in Grothendieck abelian categories:
\begin{prop} \label{lem:zariski:derived-cat}
    There is a functor $\Cat D \colon \Groth \to \PrSt$ 
    which sends a Grothendieck abelian category to its derived $\infty$-category,
    and a colimit-preserving exact functor
    to its derived functor.

    Similarly, there is a functor $\Cat D^+ \colon \Groth \to \CatInf$ 
    that is the restriction of $\Cat D$ to bounded above objects.
\end{prop}
\begin{proof}
    Write $\PreSt \subset \PrL$ for the full subcategory consisting of those 
    presentable categories that are prestable, cf.\ \cite[Definition C.1.2.1]{sag}.
    We first construct a functor 
    \begin{equation*}
        \Cat D_{\ge 0} \colon \Groth \to \PreSt
    \end{equation*}
    which sends a Grothendieck abelian category $\Cat A$
    to the presentable prestable category $\Cat D(\Cat A)_{\ge 0}$,
    i.e.\ the connective part of the derived $\infty$-category.
    Recall from \cite[Proposition C.5.4.5]{sag}
    that there is a fully faithful functor $M \colon \Groth \to \Groth^{\mathrm{lex}, \mathrm{sep}}_\infty$,
    where the right-hand category denotes the category of separated Grothendieck prestable 
    categories, whose morphisms are colimit-preserving exact functors.
    We let $\Cat D_{\ge 0}$ be the composition of $M$ with the canonical forgetful functor $\Groth^{\mathrm{lex}, \mathrm{sep}}_\infty \to \PreSt$.
    It now follows from \cite[Propositions C.5.4.5 and C.5.3.2]{sag}
    that $\Cat D_{\ge 0}$ sends a Grothendieck abelian category $\Cat A$ to the connective part of the derived $\infty$-category of $\Cat A$.

    If $\Cat A$ is a Grothendieck abelian category,
    then its derived $\infty$-category $\Cat D(\Cat A)$ is the stabilization of $\Cat D(\Cat A)_{\ge 0}$
    by \cite[Remark C.1.2.10]{sag} since the t-structure on the derived $\infty$-category of $\Cat A$ is right-complete.
    Hence, we can define $\Cat D$ as the composition of $\Cat D_{\ge 0}$ with the stabilization functor 
    $\Sp(-) \colon \PrL \to \PrSt$.

    To get the description on morphisms, note that for each morphism $F$ in $\Groth$ the functor
    $\Cat D_{\ge 0}(F)$ is exact and colimit-preserving,
    and hence preserves categorical $n$-truncated and $n$-connective objects.
    Thus, it is t-exact on the stabilization.
    Therefore, by the universal property of the derived $\infty$-category \cite[Proposition A.2]{cartmod},
    it suffices to show that the restriction of $\Cat D(F)$ to the heart is given by $F$,
    which is obvious from the above.

    To get $\Cat D^+$, one uses that all involved derived functors are t-exact,
    and thus everything restricts to bounded above objects.
\end{proof}

\begin{defn} \label{defn:zariski:site}
    Let $X$ be a qcqs $\Fp$-scheme.
    We write $X_{\zar}$ for the (small) full subcategory of $\operatorname{Sch}_{/X}$ 
    spanned by those schemes $Y \to X$ that can be written as 
    $Y = \sqcup_{i \in I} Y_i \to X$, where $I$ is finite,
    and $Y_i \to X$ is the inclusion of a quasi-compact open subset of $X$.
    In particular, since every $U \in X_{\zar}$ is a finite disjoint union of quasi-compact open subset of the qcqs $X$, it is clear that 
    $U$ is moreover quasi-separated (as open immersions are quasi-separated \cite[\href{https://stacks.math.columbia.edu/tag/01L7}{Tag 01L7}]{stacks-project},
    and quasi-separated morphisms are stable under composition).
    We equip $X_{\zar}$ with the Grothendieck topology where covers are 
    given by jointly surjective families $\{ U_i \to U\}_{i \in I}$.
    To see that this in fact defines a Grothendieck topology, we have to check that $X_{\zar}$
    is closed under pullbacks. This is true since the intersection of quasi-compact 
    opens is quasi-compact since $X$ was assumed to be quasi-separated.
\end{defn}

\begin{rmk} \label{rmk:zariski:qcqs-of-transition-maps}
    This version of the Zariski site is nonstandard.
    Since affines are quasi-compact,
    it is still true that equivalences of sheaves can be checked on affine opens if the scheme is geometric, see \cref{def:zariski:geometric,lem:zariski:check-on-affines} below.
    We use this version of the site to avoid problems with quasi-coherence and flat base change.
    Indeed, e.g. \cite[\href{https://stacks.math.columbia.edu/tag/02KH}{Tag 02KH}]{stacks-project} 
    needs the morphism $f$ to be qcqs.
    Since in our case $f$ will always be one of the open subset inclusions $U \hookrightarrow V$ 
    in $X_{\zar}$, this follows from 
    \cite[\href{https://stacks.math.columbia.edu/tag/01KV}{Tag 01KV}]{stacks-project} and 
    \cite[\href{https://stacks.math.columbia.edu/tag/03GI}{Tag 03GI}]{stacks-project} 
    since $U$ and $V$ were assumed to be qcqs.
\end{rmk}

\begin{rmk}
    If $X$ is a qcqs scheme such that any open subset of $X$ is quasi-compact 
    (e.g.\ $X$ Noetherian), then $X_{\zar}$ is 
    (the finite disjoint union completion of) the usual small Zariski site of $X$.
\end{rmk}

\begin{lem} \label{lem:zariski:check-on-affines}
    Let $X$ be a qcqs scheme and let $\Cat C$ be a complete $\infty$-category.
    Let $\Cat F, \Cat G \colon \op X_{\zar} \to \Cat C$ be Zariski sheaves,
    and let $f \colon \Cat F \to \Cat G$ be a morphism.
    Suppose that for every affine open $U \subset X$ 
    the morphism $f_U \colon \Cat F(U) \to \Cat G(U)$ is an equivalence.
    Then $f$ is an equivalence.
\end{lem}
\begin{proof}
    We have to show that for every $V \in X_{\zar}$ the morphism $f_V \colon \Cat F(V) \to \Cat G(V)$ 
    is an equivalence. Let $\Cat U = \{U_i\}_i$ be a finite Zariski cover of $V$, where the $U_i$ are affine.
    
    Assume first that $V$ has affine diagonal. Thus, any intersection of the $U_i$ 
    is also affine.
    In particular, the \v{C}ech nerve $\Cat U^\bullet$ of the cover consists only of affine schemes.
    Since $\Cat F$ is a sheaf, $\Cat F(V) \cong \lim_{\Delta} \Cat F(\Cat U^\bullet)$,
    and similarly for $\Cat G$.
    By assumption, $f$ induces an equivalence
    on the limit diagram and therefore also on the limit.

    Now, let $V$ be an arbitrary object of $X_{\zar}$. Note that any intersection of the $U_i$ 
    is quasi-affine (i.e.\ an open subset of an affine scheme). In particular, it is separated 
    and thus has affine diagonal. Arguing as above, we see that $f$ induces an equivalence 
    on the limit diagrams (using the first part of the proof), and hence on the limit.
\end{proof}

We now show that certain constructions like quasi-coherent sheaves satisfy Zariski descent.

\begin{lem} \label{lem:zariski:qcoh}
    Let $X$ be a qcqs $\Fp$-scheme.
    There is a functor $\QCoh{-} \colon \op{X}_{\zar} \to \Groth$
    which on objects sends an open $U \subset X$ to the category 
    of quasi-coherent $\mathcal{O}_U$-modules.
    Similarly, the associated morphism to an inclusion $U \subset V$ of open subsets of $X$
    is the restriction of quasi-coherent $\mathcal{O}_V$-modules to $U$.

    Moreover, $\QCoh{-}$ is a Zariski sheaf.
\end{lem}
\begin{proof}
    We let $\QCoh{-} \colon \op{X}_{\zar} \to \CatInf$ be the straightening of the 
    cartesian fibration from \cite[\href{https://stacks.math.columbia.edu/tag/03YL}{Tag 03YL}]{stacks-project} (restricted to $\op{X}_{\zar}$).
    Note that by construction, we have that $\QCoh{U}$ is just the Grothendieck abelian category of 
    quasi-coherent sheaves on $U$.
    Moreover, on inclusions of open subsets $U \subset V$, $\QCoh{U \hookrightarrow V}$ is the restriction of quasi-coherent sheaves.
    Note that these restriction functors are exact as $U \hookrightarrow V$
    is an open immersion and so in particular flat.
    Furthermore, the restriction functors are also colimit-preserving as they have right adjoints
    given by pushforward (note that here we need that the morphisms $U \hookrightarrow V$ are qcqs by \cref{rmk:zariski:qcqs-of-transition-maps}).
    In particular, we see that everything factors through $\Groth$.
    That the functor $\QCoh{-}$ is a Zariski sheaf follows from \cite[\href{https://stacks.math.columbia.edu/tag/03YM}{Tag 03YM}]{stacks-project}.
    Note that the functor actually lands in the $(2,1)$-category of ordinary categories,
    hence the result from the stacks project applies.
\end{proof}

\begin{lem} \label{lem:zariski:qcoh-frob}
    Let $X$ be a qcqs $\Fp$-scheme.
    The functor $\QCoh{-} \colon \op{X}_{\zar} \to \Groth$ from \cref{lem:zariski:qcoh}
    admits a lift $\QCoh{-}^{\frob}$ to $\End{\Groth}$, i.e.\ there exists a diagonal morphism in the diagram
    \begin{center}
        \begin{tikzcd}
            & \End{\Groth} \ar[d] \\
            \op{X}_{\zar} \ar[ur, dashed, bend left, "{\QCoh{-}^{\frob}}"] \ar[r, "{\QCoh{-}}"'] &\Groth
        \end{tikzcd}
    \end{center}
    such that for any open $U \subset X$, the underlying functor of $\QCoh{U}^{\frob}$ 
    is the Frobenius pushforward $F_* \colon \QCoh{U} \to \QCoh{U}$.
\end{lem}
\begin{proof}
    By the universal property of $\End{\Groth}$ and adjunction, this boils down to giving a natural transformation
    $\QCoh{-} \to \QCoh{-}$,
    or, by unstraightening \cite[§3.2]{highertopoi}, to giving a morphism of cartesian fibrations
    \begin{equation*}
        \int \QCoh{-} \to \int \QCoh{-}.
    \end{equation*}
    These are both ordinary categories, so we can write down by hand what this morphism is:
    On objects, it sends a pair $(U, \Cat F)$, where $U \subset X$ is open and $\Cat F \in \QCoh{U}$,
    to the pair $(U, F_* \Cat F)$.
    It sends a morphism $(\iota, \phi) \colon (U, \Cat F) \to (V, \Cat G)$,
    where $\iota \colon U \hookrightarrow V$ is an inclusion of open subsets, 
    and $\phi \colon \Cat G \to \iota_* \Cat F$ is a morphism of quasi-coherent sheaves,
    to the morphism $(\iota, \phi')$,
    where $\phi' \colon F_* \Cat G \xrightarrow{F_* \phi} F_* \iota_* \Cat F \cong \iota_* F_* \Cat F$.
    Here, the equivalence exists because already $F\iota = \iota F$.
    We check that this actually defines a functor:
    So let $(\iota, \phi) \colon (U, \Cat F) \to (V, \Cat G)$
    and $(j, \psi) \colon (V, \Cat G) \to (W, \Cat H)$ be two morphisms.
    We have to see that the outer diagram commutes:
    \begin{center}
        \begin{tikzcd}
            F_* \Cat H \ar[r, "F_* \psi"] \ar[dd, "{F_*(j_* \phi \circ \psi)}"'] & F_* j_* \Cat G \ar[r, "Ex_{*,*}"] \ar[d, "F_* j_* \phi"'] &j_* F_* \Cat G \ar[d, "j_* F_* \phi"] \\
            &F_* j_* \iota_* \Cat F \ar[r, "Ex_{*,*}"'] & j_* F_* \iota_* \Cat F \ar[d, "Ex_{*,*}"] \\
            F_* j_* \iota_* \Cat F \ar[ur, equal] \ar[rr, "Ex_{*,*}"'] && j_* \iota_* F_* \Cat F \rlap{.}
        \end{tikzcd}
    \end{center}
    Here, the small square commutes by naturality of the exchange transformation,
    whereas the left trapezoid commutes by functoriality, and the right trapezoid 
    commutes by the definition of the exchange transformations.

    It suffices to show that the functor is a morphism of cartesian fibrations.
    So let $(\iota, \phi) \colon (U, \Cat F) \to (V, \Cat G)$ be a cartesian edge.
    Using the description of cartesian edges from \cite[\href{https://stacks.math.columbia.edu/tag/04U3}{Lemma 04U3}]{stacks-project}
    we know that the adjoint morphism $\tilde{\phi} \colon \iota^* \Cat G \to \Cat F$ is an equivalence,
    and we have to show that the same is true for the image of $(\iota, \phi)$.
    Hence, we have to see that the adjoint of $\phi' \colon F_* \Cat G \xrightarrow{F_* \phi} F_* \iota_* \Cat F \cong \iota_* F_* \Cat F$
    is an isomorphism. Consider the following diagram:
    \begin{center}
        \begin{tikzcd}
            \iota^* F_* \Cat G \ar[d, "Ex^*_*"'] \ar[r, "\iota^* F_* \phi"] &\iota^* F_* \iota_* \Cat F \ar[d, "Ex_*^*"] \ar[r, "Ex_{*,*}"] &\iota^* \iota_* F_* \Cat F \ar[d, "\mathrm{counit}"] \\
            F_* \iota^* \Cat G \ar[r, "F_* \iota^* \phi"']&F_* \iota^* \iota_* \Cat F \ar[r, "\mathrm{counit}"'] &F_* \Cat F\rlap{.}
        \end{tikzcd}
    \end{center}
    The left square commutes by naturality of the exchange transformation, 
    whereas the right square commutes by definition of the vertical exchange transformation as adjoint of the map
    $F_* \xrightarrow{\mathrm{unit}} F_*\iota_*\iota^* \cong \iota_* F_* \iota^*$.
    The composition on the top and right is the adjoint of $\phi'$.
    The left vertical arrow is an isomorphism by \cite[Lemma 2.2.1]{blickle2013cartier},
    whereas the bottom composition is $F_*$ applied to the adjoint of $\phi$,
    which is an isomorphism by assumption.
    This proves the claim.
\end{proof}

\begin{lem} \label{lem:zariski:qcoh-frob-sheaf}
    Let $X$ be a qcqs $\Fp$-scheme.
    The functor
    \begin{equation*}
        \QCoh{-}^{\frob} \colon \op{X}_{\zar} \to \End{\Groth}
    \end{equation*}
    from \cref{lem:zariski:qcoh-frob} is a Zariski sheaf.
\end{lem}
\begin{proof}
    We have seen in \cref{lem:zariski:qcoh} that $\QCoh{-}$ is a Zariski sheaf.
    Since $\End{\Groth} \to \Groth$ is conservative and preserves limits,
    e.g.\ by \cite[Proposition 2.6 (b) and (c)]{cartmod},
    it follows immediately that also $\QCoh{-}^{\frob}$ is a Zariski sheaf.
\end{proof}

\begin{prop}[Zariski descent of derived categories] \label{lem:zariski:descent-on-derived}
    Let $X$ be a qcqs $\Fp$-scheme.
    Let $\Cat F \colon \op{X}_{\zar} \to \Groth$ 
    be a Zariski sheaf of Grothendieck abelian categories,
    such that $\Cat F$ satisfies base change in the following sense:
    If
    \[\begin{tikzcd}
        U_1 \ar[r, "g'"] \ar[d, "f'"'] \cartsymb & U_2 \ar[d, "f"] \\
        U_3 \ar[r, "g"'] & U_4
    \end{tikzcd}\]
    is a cartesian square in $X_{\zar}$ (in particular, all the arrows are (disjoint unions of) qcqs open immersions),
    then applying the functor $\Cat D^+ \circ \Cat F$ (using \cref{lem:zariski:derived-cat}) yields a horizontally right adjointable square
    \[\begin{tikzcd}
        \Cat D^+(\Cat F(U_4)) \ar[r, "g^*"] \ar[d, "f^*"'] & \Cat D^+(\Cat F(U_3)) \ar[d, "f'^*"] \\
        \Cat D^+(\Cat F(U_2)) \ar[r, "g'^*"'] & \Cat D^+(\Cat F(U_1)) \rlap{,}
    \end{tikzcd}\]
    i.e.\ the functors $g^*$ and $g'^*$ admit right adjoints $Rg_*$ and $Rg'_*$, respectively,
    and the canonical base change map $f^* Rg_* \to Rg'_* f'^*$ is an equivalence.

    Then $\Cat D^+\circ\Cat F \colon \op{X}_{\zar} \to \CatInf$ 
    is a Zariski sheaf of stable $\infty$-categories.
\end{prop}
\begin{proof}
    It suffices to show that for every Zariski cover $\Cat U = \{U_i \hookrightarrow U\}_i$ in $X$ the functor 
    $\Cat D^+$ preserves the limit diagram $\Cat F(U) \to \Cat F(\check{C}(\Cat U)^\bullet)$,
    where $\check{C}(\Cat U)^\bullet$ is the \v{C}ech nerve of the cover.
    For notational convenience, we write $\Cat A^n \coloneqq \Cat F(\check{C}(\Cat U)^n)$,
    so that $\Cat F(U) \eqqcolon \Cat A^{-1} \cong \lim_{n \in \Delta} \Cat A^n$.
    The proposition then follows from \cite[Proposition A.4.23]{heyer20246} if we can show the following statements:
    \begin{enumerate}[label=\textbf{(\alph*)}]
        \item $\Cat A^n$ is a Grothendieck abelian category for each $n \in \Delta$.
        \item For every $\alpha \colon [n] \to [m] \in \Delta$,
            the functor $\alpha^* \colon \Cat A^n \to \Cat A^m$ is a exact functor admitting a right adjoint $\alpha_*$.
        \item For $d^i \colon [n] \to [n+1]$ the $i$-th coface map, and $\alpha \colon [n] \to [m] \in \Delta$,
            the following diagram commutes:
            \begin{center}
                \begin{tikzcd}
                    \Cat D^+ (\Cat A^{n+1}) \ar[r, "Rd^0_*"] \ar[d, "\alpha'^*"'] &\Cat D^+(\Cat A^n) \ar[d, "\alpha^*"] \\
                    \Cat D^+ (\Cat A^{m+1}) \ar[r, "Rd^0_*"'] &\Cat D^+(\Cat A^m) \rlap{,}
                \end{tikzcd}
            \end{center}
            where $\alpha'$ is defined as $\id{[0]} \star\alpha$ under the identification $[n+1]=[0]\star[n]$ and similarly for $m$.
        \item The functor $d^{1*} \colon \Cat A^0 \to \Cat A^1$ sends injective objects to $d^0_*$-acyclic objects.
    \end{enumerate}
    The statements \textbf{(a)} and \textbf{(b)} are clear since $\Cat A^\bullet$ is a diagram in $\Groth$.
    We proceed by showing \textbf{(c)}. By definition, all functors in the square are induced by open immersions of the form
    \begin{equation*}
        \bigsqcup_{i_0, \dots, i_m} \bigcap_{k = 0}^{m} U_{i_k} \to \bigsqcup_{i_0, \dots, i_n} \bigcap_{k = 0}^{n} U_{i_k}.
    \end{equation*}
    Hence, the result follows from our base change assumption.

    We end the proof by showing \textbf{(d)}.
    Let $M \in \Cat F(\sqcup_i U_i)$ be an injective object.
    We have to see that $d^{1*} M$ is $d^0_*$-acyclic.
    For this, it is enough to see that the following diagram commutes:
    \begin{center}
        \begin{tikzcd}
            \Cat D^+ (\Cat A^0) \ar[r, "R\iota_*"]\ar[d, "d^{1*}"'] &\Cat D^+(\Cat A^{-1})\ar[d, "\iota^*"] \\
            \Cat D^+ (\Cat A^1) \ar[r, "Rd^0_*"']  &\Cat D^+(\Cat A^0)\rlap{,}
        \end{tikzcd}
    \end{center}
    where $\iota \colon \sqcup_i U_i \to U$ is the canonical map.
    Indeed, if $M \in \Cat D^+ (\Cat A^0)$ is injective,
    we want to know that $Rd_*^0 d^{1*} M$ is concentrated in degree $0$.
    This follows from commutativity, as $M$ is $R\iota_*$-acyclic (as it is injective),
    and $\iota^*$ is t-exact (since it was exact on abelian categories).
    But now the commutativity of the diagram is proven in exactly the same way as in \textbf{(c)}.
\end{proof}

In order to use the above proposition in the setting of Frobenius modules, we have to make sure 
that they satisfy (a weak version of) flat base change. 
For this to make sense note that each qcqs morphism of schemes $f\colon X \to Y$ induces a left exact pushforward functor
\begin{equation*}
    f_* \colon \Frob{\QCoh{X}}{F_*} \to \Frob{\QCoh{Y}}{F_*}
\end{equation*}
on Frobenius modules via the functoriality of $\FrobSingle{-}$, cf.\ \cref{lem:functoriality:frob-functor}.
This in turn induces a right derived functor
\begin{equation*}
    Rf_* \colon \Cat D^+(\Frob{\QCoh{X}}{F_*}) \to \Cat D^+(\Frob{\QCoh{Y}}{F_*}).
\end{equation*} 
Moreover, the usual pullback functor $f^* \colon \QCoh{Y} \to \QCoh{X}$ induces a functor
\begin{equation*}
    f^* \colon \Frob{\QCoh{Y}}{F_*} \to \Frob{\QCoh{X}}{F_*}
\end{equation*}
via \cite[Definition 4.1.1]{bockle2014cohomological} (note that in \loccit Frobenius modules are called $\tau$-sheaves, and their categories are equivalent by \cite[Corollary 2.10]{cartmod}).
The proof of \cite[Proposition 4.4.5]{bockle2014cohomological} shows that the pullback functor $f^*$ is left adjoint to $f_*$.

We also need the following definition:
\begin{defn} \label{def:zariski:geometric}
    Let $X$ be a scheme. We say that $X$ is \emph{geometric} 
    if $X$ is quasi-compact and has affine diagonal (the latter is also called 
    \emph{semi-separated} in the literature).
\end{defn}

\begin{lem}[Flat base change] \label{lem:zariski:flat-base-change-qcoh}
        Let
    \[\begin{tikzcd}
        X' \ar[r, "g'"] \ar[d, "f'"'] \cartsymb & X \ar[d, "f"] \\
        S' \ar[r, "g"'] & S
    \end{tikzcd}\]
    be a cartesian diagram of geometric $\Fp$-schemes and qcqs morphisms such that $g$ is moreover flat.
    Let $M \in \QCoh{X}$.
    Then the base change map $g^* Rf_* M \to Rf'_* g'^* M$ is an equivalence in $\Cat D^+ (\QCoh{S'})$.
\end{lem}
\begin{proof}
    Since all schemes involved are geometric, it follows from \cite[\href{https://stacks.math.columbia.edu/tag/08DB}{Tag 08DB}]{stacks-project}
    that we can also work in $\Cat D^+_{\mathrm{qc}}(X)$, the full subcategory 
    of $\Cat D^+(\mathrm{Mod}_{\Cat O_X})$ (the derived category of $\Cat O_X$-modules) 
    spanned by the complexes with quasicoherent cohomology (and similarly for $X'$, $S'$ and $S$).
    Moreover, the pullback and pushforward functors are compatible with this equivalence, see 
    \cite[\href{https://stacks.math.columbia.edu/tag/0CRX}{Tag 0CRX}]{stacks-project} for the pushforward functors, and note that any square with horizontal equivalences is autmoatically vertically left-adjointable 
    to get the compatibility with the pullback functors.
    Now, we may apply the flat base change theorem \cite[\href{https://stacks.math.columbia.edu/tag/02KH}{Tag 02KH}]{stacks-project}.
\end{proof}

\begin{lem} \label{lem:zariski:flat-base-change-frob-modules}
    Let
    \[\begin{tikzcd}
        X' \ar[r, "g'"] \ar[d, "f'"'] \cartsymb & X \ar[d, "f"] \\
        S' \ar[r, "g"'] & S
    \end{tikzcd}\]
    be a cartesian diagram of geometric $\Fp$-schemes and qcqs morphisms such that $g$ is moreover flat.
    Let $M \in \Cat D^+ (\Frob{\QCoh{X}}{F_*})$.
    Then the base change map $g^* Rf_* M \to Rf'_* g'^* M$ is an equivalence in $\Cat D^+ (\Frob{\QCoh{S'}}{F_*})$.
\end{lem}
\begin{proof}
    First note that the functors $g^*$ and $g'^*$ are exact because $g$, and therefore also $g'$, are flat, respectively, 
    and because the forgetful functor $U$ is conservative and exact by \cite[Corollary 2.8 (b) and (h)]{cartmod}.
    Thus, \cite[\href{https://stacks.math.columbia.edu/tag/09T5}{Tag 09T5}]{stacks-project} implies that $g^*$ is left adjoint to $Rg_*$,
    and similarly for $g'$.
    Using this we can define the base change map as being adjoint to the map 
    $Rf_* M \xrightarrow{u} Rf_* Rg'_* g'^* M \cong Rg_* Rf'_* g'^* M$,
    where $u$ denotes the unit of the adjunction and the equivalence follows from the commutativity of the diagram.
    We have to show that this map is an equivalence, which we do after applying the conservative t-exact forgetful functor $\Cat D^+(U) \colon \Cat D^+ (\Frob{\QCoh{S'}}{F_*}) \to \Cat D^+(\QCoh{S'})$ 
    (for the conservativity see \cite[Corollary A.5]{cartmod}).

    Assume first that $M \in \Cat D^+(\Frob{\QCoh{X}}{F_*})^\heartsuit$.
    The forgetful functor commutes with the pullback functors by definition,
    and with the derived pushforward functors by the version of \cite[Proposition 6.4.2]{bockle2014cohomological} for Frobenius modules (here we use that the schemes are geometric which implies
    that the derived pushforward functors can be computed using a \v{C}ech resolution, cf.\ \cite[\href{https://stacks.math.columbia.edu/tag/01XL}{Tag 01XL}]{stacks-project}).
    This means, we have to show that the map $g^* Rf_* UM \to Rf'_* g'^* UM$ is an equivalence in $\Cat D^+(\QCoh{S'})$. 
    But this is exactly \cref{lem:zariski:flat-base-change-qcoh}.

    Since all the functors are exact, by a standard devissage argument,
    we immediately get the result for $M \in \Cat D^\flat(\Frob{\QCoh{X}}{F_*})$.
    Suppose now that $M \in \Cat D^+(\Frob{\QCoh{X}}{F_*})$.
    By separatedness of the t-structures, it suffices to show that 
    $\pi_n(g^* Rf_* UM) \to \pi_n(Rf'_* g'^* UM)$ is an equivalence for all $n \in \Z$.
    For this, consider the following diagram:
    \begin{center}
        \begin{tikzcd}
            \tau_{\ge n} g^* Rf_* UM \ar[r] &\tau_{\ge n} Rf'_* g'^* UM \\
            \tau_{\ge n} g^* Rf_* U \tau_{\ge n} M \ar[u, "\cong"] \ar[r, "\cong"] &\tau_{\ge n} Rf'_* g'^* U \tau_{\ge n} M \ar[u, "\cong"] \rlap{,}
        \end{tikzcd}
    \end{center}
    where the vertical arrows are equivalences 
    since $g^* Rf_* U$ and $Rf'_* g'^* U$ are left t-exact (in fact, $g^*$, $g'^*$ and $U$ are t-exact,
    and $Rf_*$ and $Rf'_*$ are right adjoints of t-exact functors).
    Indeed, this can be checked on homotopy groups;
    by definition they vanish below $n$, and in degrees $\ge n$ they agree 
    by the long exact sequence associated to the fiber sequence 
    \begin{equation*}
        \tau_{\ge n} g^* Rf_* U \tau_{\ge n} M \to \tau_{\ge n} g^* Rf_* U M \to \tau_{\ge n} g^* Rf_* U \tau_{\le n-1} M,
    \end{equation*}
    using that $g^* Rf_* U$ is left t-exact, and similarly for the right vertical morphism.
    But the lower horizontal arrow is an equivalence by the above, as $\tau_{\ge n} M$
    is bounded.
    This immediately implies the result.
\end{proof}

\begin{prop} \label{lem:zariski:left-sep-derived-category-sheaf}
    Let $X$ be a geometric $\Fp$-scheme.
    The presheaf of stable $\infty$-categories 
    \begin{equation*}
        \Cat D^+(\FrobSingle{\QCoh{-}^{\frob}}) \colon \op{X}_{\zar} \to \CatInf
    \end{equation*}
    is a Zariski sheaf.
    Here, $\QCoh{-}^{\frob}$ is the functor from \cref{lem:zariski:qcoh-frob}.
\end{prop}
\begin{proof}
    By \cref{lem:zariski:qcoh-frob-sheaf}, the presheaf $\QCoh{-}^{\frob}$ is a Zariski sheaf.
    Moreover, the functor $\FrobSingle{-}$ preserves limits by \cref{lem:functoriality:frob-functor-on-groth-prst-limits},
    hence also the presheaf $\FrobSingle{\QCoh{-}^{\frob}}$ is a Zariski sheaf (as Zariski sheaves are defined using a limit condition).
    Therefore, the proposition follows from \cref{lem:zariski:descent-on-derived},
    using \cref{lem:zariski:flat-base-change-frob-modules}.
\end{proof}

\begin{lem} \label{lem:zariski:left-sep-derived-category-qcoh-sheaf}
    Let $X$ be a geometric $\Fp$-scheme.
    The presheaf of stable $\infty$-categories
    \begin{equation*}
        \Cat D^+(\QCoh{-}) \colon \op{X}_{\zar} \to \CatInf
    \end{equation*}
    is a Zariski sheaf.
\end{lem}
\begin{proof}
    Since $\QCoh{-}$ is a Zariski sheaf by \cref{lem:zariski:qcoh},
    this follows from \cref{lem:zariski:descent-on-derived},
    using \cref{lem:zariski:flat-base-change-qcoh}
    (which applies 
    since all involved morphisms are qcqs, cf.\ \cref{rmk:zariski:qcqs-of-transition-maps}).
    Note that this gives that the base change map is an equivalence on objects in the heart,
    but one can deduce the result for every bounded above object by an analogous argument 
    as in the proof of \cref{lem:zariski:flat-base-change-frob-modules}.
\end{proof}

\begin{lem} \label{lem:zariski:end-left-sep-derived-category-sheaf}
    Let $X$ be a geometric $\Fp$-scheme.
    The presheaf of stable $\infty$-categories
    \begin{equation*}
        \End{\Cat D^+}(\QCoh{-}^{\frob}) \colon \op{X}_{\zar} \to \End{\CatInf}
    \end{equation*}
    is a Zariski sheaf.
\end{lem}

\begin{proof}
    By \cref{lem:zariski:qcoh-frob-sheaf}, the presheaf $\QCoh{-}^{\frob}$ is a Zariski sheaf.
    Thus, it suffices to show that for every (quasi-compact) Zariski cover $\Cat U = \{U_i \hookrightarrow U\}_i$ in $X_{\zar}$, the functor 
    $\End{\Cat D^+}$ preserves the limit diagram $\QCoh{U}^{\frob} \to \QCoh{\check{C}(\Cat U)^\bullet}^{\frob}$,
    where $\check{C}(\Cat U)^\bullet$ is the \v{C}ech nerve of the cover.
    By \cref{lem:functoriality:end-limit-preserving}, it is enough to show that the functor $\Cat D^+$ preserves the limit diagram $\QCoh{U} \to \QCoh{\check{C}(\Cat U)^\bullet}$.
    This was shown in \cref{lem:zariski:left-sep-derived-category-qcoh-sheaf}.
\end{proof}

\begin{lem} \label{lem:zariski:canonical-map}
    Let $X$ be a geometric $\Fp$-scheme.
    There is a morphism of presheaves on $X_{\zar}$
    \begin{equation*}
        \Phi \colon \Cat{D}^+(\FrobSingle{\QCoh{-}^{\frob}}) \to \FrobSingle{\End{\Cat{D}^+}(\QCoh{-}^{\frob})}
    \end{equation*}
    such that for each (quasi-compact) open $U \subset X$,
    the map $\Phi_U$ can be identified with a t-exact functor
    \begin{equation*}
        \Phi_U \colon \Cat{D}^+(\Frob{\QCoh{U}}{F_*}) \to \Frob{\Cat D^+(\QCoh{U})}{\Cat D^+(F_*)},
    \end{equation*}
    where the target is equipped with the induced t-structure from \cite[Proposition 3.3]{cartmod}.
    If $U = \operatorname{Spec}(R)$ is affine, then this map can be identified with a map
    \begin{equation*}
        \Phi_R \colon \Cat{D}^+(\Frob{\ModDisc{R}}{F_*}) \to \Frob{\Cat D^+(\ModDisc{R})}{\Cat D^+(F_*)},
    \end{equation*}
    which is the restriction of the equivalence from \cref{lem:affine:main-thm}.
\end{lem}
\begin{proof}
    We will construct a map from the bottom-left to the top-right composition of the below square. Then precomposition with $\QCoh{-}^{\frob}$ gives the desired map.
    \begin{center}
        \begin{tikzcd}
            \End{\Groth} \ar[d, "{\Frob{-}{-}}"'] \ar[rr, "\End{\Cat D^+}"] &&\End{\CatInf} \ar[d, "{\Frob{-}{-}}"]\\
            \Groth \ar[rr, "\Cat D^+"'] &&\CatInf
        \end{tikzcd}
    \end{center}
    Let $\Cat A \in \Groth$, and let $F \colon \Cat A \to \Cat A$ be a colimit-preserving and exact endofunctor.
    Hence, unwinding the definitions, we have to find a map
    \begin{equation*}
        \Phi_{(\Cat A, F)} \colon \Cat D^+(\Frob{\Cat A}{F}) \to \Frob{\Cat D^+(\Cat A)}{\Cat D^+(F)}
    \end{equation*}
    natural in the pair $(\Cat A, F)$.
    Using the universal property of the pullback, we let $\Phi_{(\Cat A, F)}$
    be the dashed morphism in the following diagram:
    \begin{center}
        \begin{tikzcd}
            \Cat D^+(\Frob{\Cat A}{F}) \ar[dr, dashed, "{\Phi_{(\Cat A, F)}}"] \ar[dd, "{\Cat D^+(\kappa)}"'] \ar[rrd, bend left, "\Cat D^+(U)"] && \\
            & \Frob{\Cat D^+(\Cat A)}{\Cat D^+(F)} \cartsymb \ar[r, "U"] \ar[d, "\kappa"'] &\Cat D^+(\Cat A) \ar[d, "{(\id{}, \Cat D^+(F))}"] \\
            \Cat D^+(\Fun{\Delta^1}{\Cat A}) \ar[r]& \Fun{\Delta^1}{\Cat D^+(\Cat A)} \ar[r, "{(s, t)}"'] &\Cat D^+(\Cat A) \times \Cat D^+(\Cat A) \rlap{.}
        \end{tikzcd}
    \end{center}
    The bottom left functor is the restriction of the functor from the discussion before \cite[Theorem 5.1]{cartmod} to the subcategories of bounded above objects.
    The commutativity of the outer solid diagram can be deduced in the exact same way 
    as the commutativity of the diagram in \cite[Theorem 5.1]{cartmod}, see the paragraph directly behind the statement of the theorem.
    As everything in this diagram is natural in $(\Cat A, F)$, the same is true for $\Phi_{(\Cat A, F)}$.
    Moreover, the t-exactness of $\Phi_{(\Cat A, F)}$ follows immediately from the definition of the induced t-structure \cite[Proposition 3.3]{cartmod},
    and the fact that $\Cat D^+(U)$ is t-exact.

    It is clear from the construction that on affine $U = \mathrm{Spec}(R)$ the functor $\Phi_U$ 
    is exactly the restriction of the equivalence from \cref{lem:affine:main-thm}.
\end{proof}

\begin{thm} \label{lem:zariski:main-thm}
    Let $X$ be a geometric $\Fp$-scheme.
    Then the canonical morphism 
    \begin{equation*}
        \Phi_X \colon \Cat{D}^+ (\Frob{\QCoh{X}}{F_*}) \xrightarrow{\simeq} \Frob{\Cat{D}^+(\QCoh{X})}{\Cat{D}^+(F_*)}
    \end{equation*}
    from \cref{lem:zariski:canonical-map} is a t-exact equivalence.
\end{thm}
\begin{proof}
    By \cref{lem:zariski:canonical-map},
    there is even a morphism of presheaves on $X_{\zar}$
    \begin{equation*}
        \Phi \colon \Cat{D}^+ (\FrobSingle{\QCoh{-}^{\frob}}) \to \FrobSingle{\End{\Cat D^+}(\QCoh{-}^{\frob})},
    \end{equation*}
    that is sectionwise t-exact. We have to show that this morphism is an equivalence on global sections.
    The source of $\Phi$ is a Zariski sheaf by \cref{lem:zariski:left-sep-derived-category-sheaf}.
    Similarly, the target of $\Phi$ is also a Zariski sheaf: 
    Indeed, the composition $\End{\Cat D^+} \, \circ \, \QCoh{-}^{\frob}$ is a Zariski sheaf by \cref{lem:zariski:end-left-sep-derived-category-sheaf}, 
    and the functor $\Frob{-}{-} \colon \End{\CatInf} \to \CatInf$ preserves limits by \cref{lem:functoriality:frob-functor-limits}.
    Thus, it follows that also the composition $\Frob{-}{-} \, \circ \, \End{\Cat D^+} \, \circ \, \QCoh{-}^{\frob}$ is a Zariski sheaf. 
    
    Hence, in order to see that $\Phi$ is an equivalence on global sections,
    it suffices to see that it is an equivalence on affine opens $U \subset X$, cf.\ \cref{lem:zariski:check-on-affines}.
    Now note that this case was shown in \cref{lem:affine:main-thm} (as we have identified the morphism $\Phi_U$ with the restriction of the morphism 
    in \loccit by \cref{lem:zariski:canonical-map}).
\end{proof}

\section{Left-completeness on geometric schemes}
In this section, we show that the derived categories of quasi-coherent modules and Frobenius modules, respectively,  
over a geometric scheme are already left-complete. This will imply our main theorem, cf.\ \cref{lem:ab4:main-thm}.

Recall from \cite[Definition 8.3 (d)]{antieau2018uniqueness} the axiom $\mathrm{AB}4^*n(\omega)$ 
for Grothendieck abelian categories, this is a weaker variant of Grothendieck's axiom $\mathrm{AB}4^*$.
If a Grothendieck abelian category $\Cat A$ satisfies $\mathrm{AB}4^*n(\omega)$ 
for some $n \ge 0$, then $\Cat D(\Cat A)$ is left-complete by 
\cite[Proposition 8.14]{antieau2018uniqueness}. Hence, we get the following:

\begin{lem} [{\cite[Theorem 1.4]{positselski2025roosaxiomholdsquasicoherent}}] \label{lem:ab4:qcoh-is-ab4n} 
    Let $X$ be a geometric scheme.
    There exists an $n \ge 0$ such that $\QCoh{X}$ is $\mathrm{AB}4^*n(\omega)$.
    In particular, the derived $\infty$-category $\Cat D(\QCoh{X})$ is left-complete.
\end{lem}

\begin{prop} \label{lem:ab4:frob-qcoh-is-ab4n}
    Let $X$ be a geometric $\Fp$-scheme.
    There exists an $n \ge 0$ such that $\Frob{\QCoh{X}}{F_*}$ is $\mathrm{AB}4^*n(\omega)$.
    In particular, the derived $\infty$-category $\Cat D(\Frob{\QCoh{X}}{F_*})$ is left-complete.
\end{prop}
\begin{proof}
    By \cref{lem:ab4:qcoh-is-ab4n} there is an $n \ge 0$ such that $\QCoh{X}$ is $\mathrm{AB}4^*n(\omega)$.
    It is enough to show that $\Frob{\QCoh{X}}{F_*}$ is also $\mathrm{AB}4^*n(\omega)$.
    Let $(M_k)_k$ be a countable family in $\Frob{\QCoh{X}}{F_*}$.
    We have to show that the product $\prod_k HM_k \in \Cat D(\Frob{\QCoh{X}}{F_*})_{\ge -n}$.
    This statement only depends on the coconnective part $\Cat D_{\le 0}(\Frob{\QCoh{X}}{F_*})$.
    Hence, using \cref{lem:zariski:main-thm}, we may work in the $\infty$-category $\Frob{\Cat D_{\le 0}(\QCoh{X})}{\Cat D_{\le 0}(F_*)}$.
    Now, by definition of the t-structure, connectivity 
    may be checked after applying the limit-preserving t-exact functor $U \colon \Frob{\Cat D_{\le 0}(\QCoh{X})}{\Cat D_{\le 0}(F_*)} \to \Cat D_{\le 0}(\QCoh{X})$.
    But $U \prod_k HM_k \cong \prod_k HUM_k \in \Cat D_{\le 0}(\QCoh{X})_{\ge -n}$
    as $\QCoh{X}$ is $\mathrm{AB}4^*n(\omega)$.
\end{proof}

To show that $\Frob{\Cat D(\QCoh{X})}{\Cat D(F_*)}$ is left-complete as well,
we need the following result:

\begin{lem}
    Let $\Cat D$ be a presentable stable $\infty$-category with a t-structure.
    Then $\Cat D$ is left-complete if and only if 
    \begin{enumerate}[label=(\alph*)]
        \item $X \to \lim_n \tau_{\le n} X$ is an equivalence for every $X \in \Cat D$, and
        \item $\tau_{\le k} \lim_n X_n \to X_k$ is an equivalence for every $k \in \Z$ and every 
            $(X_n)_n \in \lim_n \Cat D_{\le n}$ (i.e.\ $X_n \in \Cat D_{\le n}$ for every $n$,
            together with compatible equivalences $\tau_{\le m} X_n \cong X_m$ for all $m \le n$).
    \end{enumerate}
\end{lem}
\begin{proof}
    By definition, $\Cat D$ is left-complete if and only if the canonical functor 
    $\Cat D \to \lim_n \Cat D_{\le n}$ is an equivalence.
    This functor has a right adjoint, given by sending $(X_n)_n$ to its limit $\lim_n X_n \in \Cat D$.
    The two maps given in the statement of the lemma are exactly the unit and counit of this adjunction.
    Hence, $\Cat D$ is left-complete if and only if those maps are equivalences.
\end{proof}

As an immediate corollary, we obtain:
\begin{cor} \label{lem:ab4:conservative-limit-preserving-reflects-left-complete}
    Let $U \colon \Cat D \to \Cat E$ be a t-exact conservative limit-preserving functor 
    of presentable stable $\infty$-categories with t-structures.
    If the t-structure on $\Cat E$ is left-complete, then the t-structure on $\Cat D$ is also left-complete.
\end{cor}

Using \cref{lem:ab4:qcoh-is-ab4n}, we can apply this result to $U \colon \Frob{\Cat D(\QCoh{X})}{\Cat D(F_*)} \to \Cat D(\QCoh{X})$, 
which is conservative and limit-preserving by \cite[Corollary 2.8 (b) and (c)]{cartmod},
and t-exact by definition of the induced t-structure \cite[Proposition 3.3]{cartmod}.
Hence, we get:

\begin{cor} \label{lem:ab4:frobD-is-left-complete}
    Let $X$ be a geometric $\Fp$-scheme. 
    The induced t-structure on $\Frob{\Cat D(\QCoh{X})}{\Cat D(F_*)}$ 
    is left-complete.
\end{cor}

Next, we state and prove our main theorem. For a geometric $\Fp$-scheme $X$, using the universal property of the pullback,
we let $\Phi_X$ be the dashed morphism in the following diagram, where we write for notational convenience $\Cat A \coloneqq \QCoh{X}$: 
\begin{center}
    \begin{tikzcd}
        \Cat D(\Frob{\Cat A}{F_*}) \ar[dr, dashed, "{\Phi_X}"] \ar[dd, "{\Cat D(\kappa)}"'] \ar[rrd, bend left, "\Cat D(U)"] && \\
        & \Frob{\Cat D(\Cat A)}{\Cat D(F_*)} \cartsymb \ar[r, "U"] \ar[d, "\kappa"'] &\Cat D(\Cat A) \ar[d, "{(\id{}, \Cat D(F_*))}"] \\
        \Cat D(\Fun{\Delta^1}{\Cat A}) \ar[r]& \Fun{\Delta^1}{\Cat D(\Cat A)} \ar[r, "{(s, t)}"'] &\Cat D(\Cat A) \times \Cat D(\Cat A) \rlap{.}
    \end{tikzcd}
\end{center}
Here, the bottom left functor is the functor from the discussion before \cite[Theorem 5.1]{cartmod}.
As in the last sections, the commutativity of the outer solid diagram can be deduced in the exact same way 
as the commutativity of the diagram in \cite[Theorem 5.1]{cartmod}, see the paragraph directly behind the statement of the theorem.

To prove the main theorem, we also need the following lemma about left-complete stable $\infty$-categories:

\begin{lem} \label{lem:ab4:plus-equivalence}
    Let $F \colon \Cat D \to \Cat E$ be a t-exact functor 
    of presentable stable $\infty$-categories with t-structures such that both $\infty$-categories are left-complete.
    If the induced functor $F^+ \colon \Cat D^+ \to \Cat E^+$ is an equivalence,
    then also $F \colon \Cat D \to \Cat E$ is an equivalence.
\end{lem}
\begin{proof}
    This is immediate since $\Cat D_{\le n} \cong \Cat D^+_{\le n} \xrightarrow[F]{\simeq} \Cat E^+_{\le n} \cong \Cat E_{\le n}$ for all $n \in \Z$,
    and $\Cat D \cong \lim_n \Cat D_{\le n}$, and similarly for $\Cat E$.
\end{proof}

\begin{thm} \label{lem:ab4:main-thm}
    Let $X$ be a geometric $\Fp$-scheme. Then the functor
    \begin{equation*}
        \Phi_X \colon \Cat{D} (\Frob{\QCoh{X}}{F_*}) \xrightarrow{\simeq} \Frob{\Cat{D}(\QCoh{X})}{\Cat{D}(F_*)}
    \end{equation*}
    described above is a t-exact equivalence,
    where we equip the target with the induced t-structure from \cite[Proposition 3.3]{cartmod}.
\end{thm}
\begin{proof}
    The t-exactness of the functor follows immediately from the definition of the induced t-structure and the t-exactness of $\Cat D(U)$.
    Thus, $\Phi_X$ restricts to a functor
    \begin{equation*}
        \Phi_X^+ \colon \Cat{D}^+ (\Frob{\QCoh{X}}{F_*}) \to \Frob{\Cat{D}^+(\QCoh{X})}{\Cat{D}^+(F_*)}.
    \end{equation*}
    By construction, $\Phi_X^+$ is the same functor as in \cref{lem:zariski:main-thm},
    where we also showed that it is an equivalence. 
    Note that by definition of the induced t-structure on $\Frob{\Cat{D}(\QCoh{X})}{\Cat{D}(F_*)}$
    the category of bounded above objects of this category is exactly given by $\Frob{\Cat{D}^+(\QCoh{X})}{\Cat{D}^+(F_*)}$.
    Hence, \cref{lem:ab4:plus-equivalence} implies that the functor
    \begin{equation*}
        \Phi_X \colon \Cat{D} (\Frob{\QCoh{X}}{F_*}) \xrightarrow{\simeq} \Frob{\Cat{D}(\QCoh{X})}{\Cat{D}(F_*)}
    \end{equation*}
    is also an equivalence, as both source and target are left-complete by \cref{lem:ab4:frob-qcoh-is-ab4n,lem:ab4:frobD-is-left-complete}.
\end{proof}

\appendix
\section{Compact projective generators}
In the literature, there are a variety of notions of compact (projective) generators 
in various settings, e.g.\ in abelian categories or in stable $\infty$-categories.
The goal of this appendix is to translate between some of those notions 
so that we may use different results from the literature.

\begin{defn} \label{defn:cpgen:abelian-generator}
    Let $\Cat A$ be a Grothendieck abelian category and $P \in \Cat A$ an object.
    We say that $P$ is a \emph{compact projective generator}
    if $\Map{\Cat A}{P}{-}$ commutes with sifted colimits and is conservative.
\end{defn}

This definition immediately implies the following:
\begin{lem} \label{lem:cpgen:tiny}
    Let $\Cat A$ be a Grothendieck abelian category and $P \in \Cat A$ a compact projective generator.
    Then $\Map{\Cat A}{P}{-}$ commutes with all colimits.
\end{lem}
\begin{proof}
    Since $\Map{\Cat A}{P}{-}$ preserves sifted colimits by assumption,
    it suffices to show that $\Map{\Cat A}{P}{-}$ preserves finite coproducts.
    This is clear since finite coproducts are finite products in abelian categories.
\end{proof}

We have a similar definition in the stable situation.
\begin{defn} \label{defn:cpgen:stable-generator}
    Let $\Cat D$ be a presentable stable $\infty$-category and $P \in \Cat D$ an object.
    We say that $P$ is a \emph{stable compact generator}
    if $\map{\Cat D}{P}{-}$ (the mapping spectrum) commutes with filtered colimits and is conservative.
\end{defn}

\begin{exmpl}
    Let $R$ be an associative ring.
    Then $R \in \ModDisc{R}$ is a compact projective generator,
    and $R \in \Mod{R}$ is a stable compact generator.
\end{exmpl}

Our goal in this section is to prove the following proposition:
\begin{prop} \label{lem:cpgen:abelian-implies-stable}
    Let $\Cat A$ be a Grothendieck abelian category, and $P \in \Cat A$ 
    a compact projective generator.
    Then $HP \in \Cat D(\Cat A)$ is a stable compact generator.
\end{prop}

For the proof, we need the following two lemmas.

\begin{lem} \label{lem:cpgen:Tohoku-generator}
    Let $\Cat A$ be a Grothendieck abelian category, and $P \in \Cat A$ 
    a compact projective generator.
    Then every object $A \in \Cat A$ admits a surjection $\oplus_I P \twoheadrightarrow A$
    for some (small) set $I$.
\end{lem}
\begin{proof}
    Consider the canonical map 
    \begin{equation*}
        \Psi \colon \bigoplus_{f \colon P \to A} P \to A.
    \end{equation*}
    We claim that this map is surjective, i.e.\ $\mathrm{coker}(\Psi) \cong 0$.
    Since $\Map{\Cat A}{P}{-}$ is conservative and preserves cokernels (by \cref{lem:cpgen:tiny}),
    it suffices to show that the induced map of abelian groups
    \begin{equation*}
        \Map{\Cat A}{P}{\bigoplus_{f \colon P \to A} P} \to \Map{\Cat A}{P}{A}
    \end{equation*}
    is surjective.
    This is clear as any $f \colon P \to A$ 
    is the image of the inclusion of the $f$-th summand.
\end{proof}

\begin{rmk}
    The last lemma shows that a compact projective generator 
    is in particular a generator in the sense of Grothendieck, see 
    \cite[Proposition 1.9.1]{grothendieck1957tohoku}.
\end{rmk}

\begin{lem} \label{lem:cpgen:gabriel}
    Let $\Cat A$ be a Grothendieck abelian category, and $P \in \Cat A$ 
    a compact projective generator.
    Then there is an equivalence $\Psi \colon \Cat A \xrightarrow{\simeq} \ModDisc{R}$,
    where $R \coloneqq \End{P}$ is the endomorphism ring of $P$,
    and $\Psi(P) = R$.
\end{lem}
\begin{proof}
    Note that $P$ is projective (in the abelian sense, i.e.\ $\Map{\Cat A}{P}{-}$ commutes with cokernels)
    by \cref{lem:cpgen:tiny}.
    Hence, in view of \cref{lem:cpgen:Tohoku-generator}, this follows from Gabriel's theorem \cite[Corollaire V.1.1]{gabriel}.
\end{proof}

\begin{proof}[Proof of \cref{lem:cpgen:abelian-implies-stable}]
    By \cref{lem:cpgen:gabriel}, we may assume that $\Cat A = \ModDisc{R}$
    for some associative ring $R$, and $P = R$.
    In particular, we know $\Cat D(\Cat A) \cong \Mod{R}$.
    Our goal is to show that $\map{\Mod{R}}{R}{-}$ is conservative and preserves filtered colimits.
    This functor can be identified with the canonical forgetful functor $\Mod{R} \to \Sp$,
    cf.\ \cref{lem:affine:maps-from-unit}.
    Hence, the result follows from \cref{lem:affine:free-forget}.
\end{proof}

We end the section by proving a stability property of generators.
\begin{lem} \label{lem:cpgen:preservation}
    Let $L \colon \Cat C \to \Cat D$ be a left adjoint functor
    of $\infty$-categories
    such that the right adjoint $R \colon \Cat D \to \Cat C$ 
    preserves colimits and is conservative.
    \begin{enumerate}
        \item[(a)] If both $\Cat C$ and $\Cat D$ are Grothendieck abelian,
            then $L$ preserves compact projective generators.
        \item[(b)] If both $\Cat C$ and $\Cat D$ are presentable stable,
            then $L$ preserves stable compact generators.
    \end{enumerate}
\end{lem}
\begin{proof}
    For the abelian case, let $P \in \Cat C$ be a compact projective generator.
    The statement then follows from the equivalence $\Map{\Cat D}{LP}{-} \cong \Map{\Cat C}{P}{R-}$
    and the properties of $R$.
    In the stable case, note that the adjunction is automatically spectrally enriched.
\end{proof}

\bibliography{main}

\end{document}